\documentclass[hidelinks,onefignum,onetabnum]{siamart251216}

\usepackage{lipsum}
\usepackage{amsfonts}
\usepackage{graphicx}
\usepackage{epstopdf}
\usepackage{algorithmic}
\ifpdf
  \DeclareGraphicsExtensions{.eps,.pdf,.png,.jpg}
\else
  \DeclareGraphicsExtensions{.eps}
\fi
\usepackage{amsmath,amsfonts,amssymb,amscd,mathrsfs}
\usepackage{calc}
\usepackage{multicol}
\usepackage{empheq}
\usepackage{amsopn}

\usepackage{arydshln}
\usepackage[caption=false]{subfig}

\newcommand{\stfD}{\mathcal{E}}

\definecolor{HsL}{RGB}{255,0,255}
\definecolor{HtL}{RGB}{90,13,67}

\newsiamremark{remark}{Remark}
\newsiamremark{hypothesis}{Hypothesis}
\crefname{hypothesis}{Hypothesis}{Hypotheses}
\newsiamthm{claim}{Claim}
\newsiamremark{fact}{Fact}
\crefname{fact}{Fact}{Facts}

\headers{MFEM for R13}{Shuang Hu, Huiteng Li and Zhenning Cai}

\title{
Stability and Convergence of Mixed Finite Elements for Linear Regularized 13-Moment Equations
\thanks{
Shuang Hu and Huiteng Li are co-first authors with equal contributions.
\funding{The work of Zhenning Cai was supported by the Academic Research Fund of the
Ministry of Education of Singapore under Grant No. A-8002392-00-00.}
}
}

\author{
Shuang Hu\thanks{
Zhejiang University, 866 Yuhangtang Road, Hangzhou, Zhejiang Province, 310058 China
(\email{hsmathna@gmail.com}).}
\and Huiteng Li\thanks{
King Abdullah University of Science and Technology (KAUST), Thuwal, 23955 Saudi Arabia
(\email{huiteng.li@kaust.edu.sa}).}
\and Zhenning Cai\thanks{ Corresponding author. 
Department of Mathematics, National University of Singapore,
Level 4, Block S17, 10 Lower Kent Ridge Road, Singapore 119076  (\email{matcz@nus.edu.sg}).}
}

\ifpdf
\hypersetup{
  pdftitle={R13FEM},
  pdfauthor={S. Hu, H. Li and Z. Cai}
}
\fi

\newsiamremark{example}{Example}
\begin{document}

\maketitle

\begin{abstract}
We present a stable and convergent mixed finite element method (MFEM) for the linear regularized 13-moment (R13) equations in rarefied gas dynamics. 
Unlike existing methods that require stabilization via penalty terms, our scheme achieves inherent stability by enriching the finite element basis with bubble functions. 
We provide a rigorous theoretical analysis, establishing second-order convergence rates in the $L^2$ norm under mild regularity assumptions. 
Beyond theoretical properties, our scheme demonstrates practical advantages over standard MFEM schemes, yielding robust numerical results even in the presence of geometric singularities.
\end{abstract}



\section{Introduction}
\label{sec:introduction}
In gas kinetic theory, moment methods have been playing important roles in the modeling and simulation of non-equilibrium flows. Since the moment method
was first introduced by Grad \cite{grad1949kinetic} in 1949, numerous moment models have been
proposed to improve the accuracy and stability of equations \cite{Levermore1996moment, Myong1999thermo, McDonald2013affordable, Struchtrup2022thermo}, among
which the regularized 13-moment (R13) equations, introduced in \cite{struchtrup2005macroscopic,struchtrup2003regularization}, have been attracting much attention in the past two decades due to their
good performance in various benchmark tests \cite{Taheri2012macroscopic, Rana2015numerical, Timokhin2017different, DeFraja2022efficient}. While the original R13 equations were devised for Maxwell molecules, they have been generalized to
other intermolecular potentials \cite{Cai2020regularized}, and in all cases, the resulting
equations enjoy both linear stability and super-Burnett order of accuracy.
Recently, stable boundary conditions have been formulated for these equations, further enhancing their potential application in practical boundary value problems
\cite{Sarna2018stable, theisen2021fenics}.
Numerical experiments indicate that R13 equations allow simulations of rarefied gas dynamics
with Knudsen number $\text{Kn} \sim 0.1$ without directly solving the
Boltzmann equation \cite{rana2013robust}.

While many moment models are gradually losing their attractiveness due to either robustness or accuracy issues, R13 equations stand as one of the models that are still being widely studied both theoretically and numerically.
One obstacle to promoting the applications of R13 equations is the large number of variables in the system, more than twice as many as Euler/Navier-Stokes-Fourier equations.
To address this issue, an easy-to-implement numerical method may be helpful.
Recently, researchers have been studying linearized R13 equations, which are accurate in the slow-flow regime, and have established several breakthroughs.
For instance, the method of fundamental solutions has been introduced in \cite{Claydon2017fundamental} and further developed in \cite{Himanshi2025generalized}.
Also, significant progress has been made on the finite element method (FEM), a field to which our work will also contribute.

Lately, the well-posedness of the steady-state linearized R13 equations for Maxwell molecules has been rigorously justified, marking another milestone in
the theoretical study of moment equations \cite{lewintan2025wellposedness} and providing a solid basis for the study of the finite element method. The weak form of R13 equations appears as a saddle-point problem, resembling the classical Stokes equations, for which the well-posedness can be established using Brezzi's theorem \cite{brezzi1974existence}. With such a structure, a natural choice of its numerical solver is the
mixed FEM. However, since the 2010s, the development of FEM
for R13 equations has undergone a bumpy road, accompanied by the formulation of
its boundary conditions. Early attempts were documented in some master theses \cite{Muller2010finite, westerkamp2012finite}, where the 13-moment equations were
separated into a stress system and a heat flux system, which were solved
independently for the purpose of a closer look at possible issues. Notably,
when the subsystems were solved using the standard Taylor-Hood elements,
oscillations were observed in velocity and temperature profiles near the
curved boundary, causing significant $L^{\infty}$ errors in the numerical
solutions. The work \cite{westerkamp2012finite} managed to
suppress the oscillation using isoparametric elements with higher-order
polynomials, and later in \cite{westerkamp2017curvature}, the
continuous interior penalty (CIP) method was adopted to obtain smooth numerical
solutions for Taylor-Hood elements. Both works \cite{westerkamp2012finite,westerkamp2017curvature}
solved only a subsystem of the R13 equations and did not report any results on the
full system.

Note that the boundary conditions used in these works were based on the
formulation introduced in \cite{torrilhon2008boundary}, which did not provide a symmetric
weak form. In \cite{rana2016thermodynamically}, thermodynamically admissible boundary conditions were formulated for R13 equations, which completely changed the landscape of FEM for R13 equations. Using the new boundary conditions, the variational form of the R13 equations was
first presented with all boundary conditions inserted in a symmetric way \cite{westerkamp2019finite}, and FEM solutions without spurious
oscillations on the curved boundary were finally obtained.

CIP stabilization, applied in all the successful FEM implementations mentioned
above, requires a few extra parameters to be tuned, which may be
problem-dependent. Can we have a solver without extra parameters? In \cite{theisen2021fenics}, with a further adjustment of the boundary conditions, desired
convergence rates were eventually achieved using $\mathbb{P}_2$-$\mathbb{P}_1$ mixed elements
without CIP stabilization. However, as we will show in \cref{sec:numerical_results}, this approach may still lead to instability in some cases. The work \cite{theisen2021fenics} 
introduces a software package known as \texttt{fenicsR13} \cite{fenicsR13repo}, which provides a large number of test cases in both 2D and 3D cases.

It has been a long journey since the first attempts of FEM for R13 equations, a
seemingly simple problem that turns out to be quite challenging. But the
question of a robust and parameter-free FEM solver for R13 equations is still
open. In this work, we would like to tackle this problem and continue the
quest. Here we summarize our contributions as follows:
\begin{itemize}
  \item We construct a stable, conforming mixed finite element method for the R13 system. Unlike existing approaches, our scheme does not require any artificial stabilization or penalty parameters.
  \item We establish the well-posedness of the discrete problem by proving the discrete inf-sup condition and deriving the a priori error estimates.
  \item We verify the validity of our scheme through various benchmarks. In particular, the scheme exhibits superior stability in problems lacking regularity, where standard schemes (e.g., Taylor-Hood) tend to fail.
\end{itemize}

While our solution is not the most computationally economical choice due to the enriched space, the overhead is relatively minor in 2D. 
More importantly, this cost is compensated for by the method's theoretical reliability. 
In contrast to CIP stabilization, 
which lacks a rigorous convergence proof, we establish the well-posedness of our discrete problem and demonstrate that the scheme achieves second-order convergence in the $L^2$ norm.

The rest of the paper is organized as follows. 
Section \ref{sec:preliminaries} introduces the necessary notation and the weak formulation of the R13 system. In Section \ref{sec:mfem}, we construct the mixed finite element scheme and define the associated finite element spaces. Section \ref{sec:stability_and_convergence} is devoted to the theoretical analysis, establishing the discrete stability and convergence of the proposed method. Numerical experiments validating these results are reported in Section \ref{sec:numerical_results}. Finally, Section \ref{sec:conclusions} summarizes the work and discusses future research directions.

\section{Preliminaries}
\label{sec:preliminaries}
\subsection{Notation}
Let $\Omega \subseteq \mathbb{R}^d$ ($d = 2,3$) be a bounded polygonal/polyhedral domain with boundary $\Gamma$.
For any subdomain $D \subseteq \Omega$ and integer $m \ge 0$, 
let $H^m(D)$ 
denote the set of all $L^2(D)$ functions 
whose weak derivatives up to order $m$ are also square-integrable.
The spaces for vector-valued or tensor-valued functions 
are given by $H^m(D; \mathbb{X}) \coloneqq H^m(D) \otimes \mathbb{X}$,
with the norm $\|\cdot\|_{m,D}$ and semi-norm $|\cdot|_{m,D}$. 
The case $m=0$ corresponds to the space of square-integrable functions,  
and we use $(\cdot,\cdot)_D$ to denote the standard $L^2$-inner product on the domain $D$ on $L^2(D)$ or $L^2(D;{\mathbb X})$. 
Denote by $H^m_0(D;{\mathbb X})$ the closure of $C^\infty_0(D;{\mathbb X})$ with respect to the norm $\|\cdot\|_{m,D}$.
Denote by $H^{-1}(D;\mathbb X)$ the dual of $H^1_0(D;\mathbb X)$ with the induced norm $\|\cdot\|_{-1,D}$.
When $D=\Omega$, the subscripts for norms, semi-norms and $L^2$-inner products are abbreviated as 
$\|\cdot\|_m$, $|\cdot|_m$ and $(\cdot,\cdot)$, respectively. 
We denote by $L^2_0$ and $\tilde{H}^1$ 
the subspaces of $L^2$ and $H^1$ 
containing functions with zero mean, 
respectively.
Throughout this paper, $A \lesssim B$ (resp., $A \gtrsim B$) denotes $A \leq CB$ (resp., $A \geq CB$) for a generic constant $C>0$ independent of the mesh size. We write $A \simeq B$ if $A \lesssim B \lesssim A$.

For the R13 system, the tensor space $\mathbb{T}$ is defined according to the dimension $d$:
\begin{equation}
\mathbb{T} \coloneqq \begin{cases}
    \mathbb{R}^{2\times2}_{\rm sym}, & d=2, \\
    \mathbb{R}^{3\times 3}_{\mathrm{stf}}, & d=3,
\end{cases}
\end{equation}
where $\mathbb{R}^{d\times d}_{\rm sym}$ and $\mathbb{R}^{d \times d}_{\mathrm{stf}}$ denote the spaces of symmetric and symmetric trace-free matrices, 
respectively. 
We define the associated projection operators $\mathrm{sym}: \mathbb{R}^{d\times d} \to \mathbb{R}^{d\times d}_{\rm sym}$ and $\mathrm{stf}: \mathbb{R}^{d\times d} \to \mathbb{R}^{d\times d}_{\rm stf}$ as:
\begin{align} 
\mathrm{sym} \, \boldsymbol{\sigma} \coloneqq \tfrac{1}{2}(\boldsymbol{\sigma} + \boldsymbol{\sigma}^T), 
\quad
\mathrm{stf} \, \boldsymbol{\sigma} \coloneqq \mathrm{sym} \, \boldsymbol{\sigma} - \tfrac{1}{d} (\mathrm{tr} \boldsymbol{\sigma}) \boldsymbol{I}_d. 
\end{align}



The definition of the symmetric trace-free projection extends to third-order tensors ($d=3$) through the operator $\mathrm{Stf}: \mathbb{R}^{3\times 3\times 3} \to \mathbb{R}_{\mathrm{stf}}^{3\times 3\times 3}$. 
For any $\boldsymbol{m} \in \mathbb R^{3\times3\times3}$,
the components of $\mathrm{Stf} \, \boldsymbol{m}$ are given by:
\begin{equation}
\label{eq:Stf}
\begin{aligned}
(\mathrm{Stf}\,\boldsymbol{m})_{ijk} 
&\coloneqq 
m_{(ijk)} - \frac{1}{5} \sum_{l=1}^{3}\left( m_{(ill)}\delta_{jk} + m_{(ljl)}\delta_{ik} + m_{(llk)}\delta_{ij}\right),
\end{aligned}
\end{equation}
where the permutation sum $m_{(ijk)}$ is defined by  
$
m_{(ijk)} := \tfrac{1}{6}( m_{ijk} + m_{jki} + m_{kij} + m_{jik} + m_{ikj} + m_{kji} )
$
and $\delta_{ij}$ is the Kronecker delta.

Finally, for brevity we define the dimension-dependent differential operator $\stfD$ as
\begin{equation}
    \stfD \coloneqq \begin{cases}
    {\rm sym}~\nabla, & d=2, \\
    {\rm stf}~\nabla, & d=3.
\end{cases}
\end{equation}

\subsection{Regularized 13-moment system}
The \textit{regularized 13-moment} (R13) system governs the evolution of the following fields: the pressure $p:\Omega\rightarrow\mathbb{R}$, the velocity $\boldsymbol{u}:\Omega\rightarrow\mathbb{R}^3$, the temperature $\theta:\Omega\rightarrow\mathbb{R}$, the heat flux $\boldsymbol{s}:\Omega\rightarrow\mathbb{R}^3$, and the symmetric trace-free stress tensor $\boldsymbol{\sigma}:\Omega\rightarrow\mathbb{R}_{\text{stf}}^{3\times3}$. Since $\boldsymbol{\sigma}$ is symmetric and trace-free, it possesses 5 degrees of freedom, resulting in a total of 13 unknowns for the system.

The R13 system consists of the \textit{conservation laws} for mass, momentum, and energy, and the \textit{evolution equations} for the higher-order moments. The conservation laws are given by:
\begin{equation}
\label{eq:conservation_part}
\begin{cases}
-\nabla\cdot\boldsymbol{u}=0,\\
\nabla p+\nabla\cdot\boldsymbol{\sigma}={0},\\
\nabla\cdot\boldsymbol{u}+\nabla\cdot\boldsymbol{s}=0,
\end{cases}
\quad\text{in }\Omega,
\end{equation}
and the evolution equations describing the stress and heat flux are:
\begin{equation}
\label{eq:evolution_part}
\begin{cases}
\frac{4}{5}\mathrm{stf}~\nabla\boldsymbol{s}+2\mathrm{stf}~\nabla\boldsymbol{u}-2\mathrm{Kn}\nabla\cdot(\mathrm{Stf}~\nabla\boldsymbol{\sigma})=-\frac{1}{\mathrm{Kn}}\boldsymbol{\sigma},\\
\frac{5}{2}\nabla\theta+\nabla\cdot\boldsymbol{\sigma}-\frac{12}{5}\mathrm{Kn}\nabla\cdot(\mathrm{stf}~\nabla\boldsymbol{s})-2\mathrm{Kn}\nabla\cdot\boldsymbol{s}=-\frac{1}{\mathrm{Kn}}\frac{2}{3}\boldsymbol{s},
\end{cases}
\quad\text{in }\Omega.
\end{equation}
Here, $\mathrm{Kn}>0$ denotes the Knudsen number.

Let $(\boldsymbol{t}_1,\boldsymbol{t}_2,\boldsymbol{n})$ be a boundary-aligned orthonormal frame with outer unit normal $\boldsymbol{n}$. 
For any $\boldsymbol{a}, \boldsymbol{b}, \boldsymbol{c} \in \{\boldsymbol{t}_1, \boldsymbol{t}_2, \boldsymbol{n}\}$, 
we define the projected components as $u_a\coloneqq\boldsymbol{u}\cdot\boldsymbol{a}$, $\sigma_{ab}\coloneqq \sum_{i,j}\sigma_{ij}a_i b_j$, 
and $m_{abc}\coloneqq\sum_{i,j,k}m_{ijk}a_i b_j c_k$.
By enforcing impermeability ($u_n=0$) consistent with Maxwell's solid-wall model, the \textit{Onsager boundary conditions} are formulated via 
higher-order moments $\bf m$, $\bf R$ and $\triangle$:
\begin{equation}
\label{eq:Onsager_BC}
\begin{cases}
u_n=0, \\
\sigma_{nt_i} = \tilde{\chi}\left(\left(u_{t_i} - u_{t_i}^W\right) + \frac{1}{5}s_{t_i} + m_{nnt_i}\right) \quad (i \in \{1,2\}), \\
R_{nt_i} = \tilde{\chi}\left(-\left(u_{t_i} - u_{t_i}^W\right) + \frac{11}{5}s_{t_i} - m_{nnt_i}\right) \quad (i \in \{1,2\}), \\
s_n = \tilde{\chi}\left(2\left(\theta - \theta^W\right) + \frac{1}{2}\sigma_{nn} + \frac{2}{5}R_{nn} + \frac{2}{15}\triangle\right), \\
m_{nnn} = \tilde{\chi}\left(-\frac{2}{5}\left(\theta - \theta^W\right) + \frac{7}{5}\sigma_{nn} - \frac{2}{25}R_{nn} - \frac{2}{75}\triangle\right), \\
\left(\frac{1}{2}m_{nnn} + m_{nt_1t_1}\right)= \tilde{\chi}\left(\frac{1}{2}\sigma_{nn} + \sigma_{t_1t_1}\right), \\
m_{nt_1t_2} = \tilde{\chi}\sigma_{nt_1t_2}.
\end{cases}
\quad \text{on } \Gamma,
\end{equation}
Here, 
$\boldsymbol{u}_{t_i}^{W}$ and $\theta^W$ denote the wall tangential velocity and wall temperature,
respectively, and $\tilde{\chi}>0$ represents the modified Maxwell accommodation factor as introduced in \cite{westerkamp2019finite}.
The closures for higher-order moments $\bf m$,
$\bf R$ and $\triangle$ in \eqref{eq:Onsager_BC} are given by
\begin{equation}
\label{eq:closure}  
\mathbf{m}
\coloneqq
-2\mathrm{Kn}\,\mathrm{Stf}\,\nabla\boldsymbol{\sigma},
\quad
\mathbf{R}
\coloneqq
-\frac{24}{5}\mathrm{Kn}\,\mathrm{stf}\,\nabla\boldsymbol{s},
\quad
\triangle\coloneqq
-12\mathrm{Kn}\nabla\cdot\boldsymbol{s}.
\end{equation}
 
Together, \cref{eq:conservation_part,eq:evolution_part,eq:Onsager_BC,eq:closure} form the complete R13 boundary value problem.

\subsection{Weak Form}
We introduce the following function spaces for the weak form:
\begin{equation}
     \label{eq:continuous_spaces}
\begin{aligned}
    &\mathrm{T}_{\boldsymbol{\sigma}} \coloneqq H^1(\Omega;\mathbb{T}),\quad \mathrm{T}_{\boldsymbol{s}} \coloneqq H^1(\Omega;\mathbb{R}^d),\quad \mathrm{T}_{p} \coloneqq \tilde{H}^1(\Omega),\\ 
    &\mathrm{W}_{\boldsymbol{u}} \coloneqq L^2(\Omega;\mathbb{R}^d),\quad \mathrm{W}_{\theta} \coloneqq L^2(\Omega).
\end{aligned}
\end{equation}

For the sake of analysis, denote by $\mathrm{T}$ and $\mathrm{W}$ composite product Hilbert spaces:
\begin{equation}
    \mathrm{T} \coloneqq \mathrm{T}_{\boldsymbol{\sigma}} \times \mathrm{T}_{\boldsymbol{s}} \times \mathrm{T}_{p}, \quad
    \mathrm{W} \coloneqq \mathrm{W}_{\boldsymbol{u}} \times \mathrm{W}_{\theta},
\end{equation}
together with their canonical norms \({\|\cdot\|}_\mathrm{T}\) and \({\|\cdot\|}_\mathrm{W}\).

Let $(\boldsymbol{\tau}, \boldsymbol{r}, q, \boldsymbol{v}, \gamma)$ be the test functions associated with the unknowns $(\boldsymbol{\sigma},\boldsymbol{s},p,\boldsymbol{u},\theta)$. 
We group these into the following tuples:
\begin{equation}
\label{eq:grouping}
\begin{aligned}
\text{Unknowns:}\quad
\boldsymbol{\mathcal{S}} \coloneqq (\boldsymbol{\sigma}, \boldsymbol{s}, p) \in \mathrm{T}, 
\quad \boldsymbol{\mathcal{U}} \coloneqq (\boldsymbol{u}, \theta) \in \mathrm{W}, \\
\text{Test Functions:}\quad
\boldsymbol{\mathcal{R}}
\coloneqq (\boldsymbol{\tau},
\boldsymbol{r}, q) \in \mathrm{T}, \quad \boldsymbol{\mathcal{V}} \coloneqq (\boldsymbol{v}, \gamma) \in \mathrm{W}.
\end{aligned}
\end{equation}

The weak formulation is stated as follows: given \(\mathcal{F} \in \mathrm{T}'\), find \(\boldsymbol{\mathcal{S}} \in \mathrm{T},\,\, \boldsymbol{\mathcal{U}} \in \mathrm{W}\) such that
\begin{equation}
\begin{aligned}
  \mathcal{A}(\boldsymbol{\mathcal{S}},\boldsymbol{\mathcal{R}})+ \mathcal{B}(\boldsymbol{\mathcal{R}},\boldsymbol{\mathcal{U}}) &= \mathcal{F}(\boldsymbol{\mathcal{R}}) \quad &\forall\, \boldsymbol{\mathcal{R}} &\in \mathrm{T},
  \\
  \mathcal{B}(\boldsymbol{\mathcal{S}},\boldsymbol{\mathcal{V}}) &= 0 \quad &\forall\, \boldsymbol{\mathcal{V}} &\in \mathrm{W}.
\end{aligned}   
\end{equation}
Here, the block operators are defined in terms of their constituent forms:
\begin{equation}
\label{eqs:def_block_operators}
\begin{cases}
\mathcal{A}(\boldsymbol{\mathcal{S}},\boldsymbol{\mathcal{R}})
\coloneqq
a(\boldsymbol{s},\boldsymbol{r})
+
c(\boldsymbol{s},\boldsymbol{\tau})
-
c(\boldsymbol{r},\boldsymbol{\sigma})
+
d(\boldsymbol{\sigma},\boldsymbol{\tau})
,
\\
\mathcal{B}(\boldsymbol{\mathcal{S}},\boldsymbol{\mathcal{V}})
\coloneqq
-
b(\gamma, \boldsymbol{s})
-
e(\boldsymbol{v},\boldsymbol{\sigma})
-
g(p,\boldsymbol{v})
,
\quad
\mathcal{F}(\boldsymbol{\mathcal{R}})
\coloneqq l_1(\boldsymbol{r})+l_2(\boldsymbol{\tau}).    
\end{cases}
\end{equation}
The linear functionals, arising from Onsager boundary conditions, are given by
\begin{equation}
\begin{cases}   
l_{1}(\boldsymbol{r}) \coloneqq - \int_{\Gamma} \theta^{\mathrm{w}} r_{n} \mathrm{d}l, \\
l_{2}(\boldsymbol{\tau}) \coloneqq - \int_{\Gamma} \sum_{i=1}^2 u_{t_i}^{\mathrm{w}} \tau_{n t_i}  \mathrm{d}l,
\end{cases}
\end{equation}
and explicit expressions for the bilinear forms 
are: 
\begin{equation}
\label{eqs:def_bilinear_forms}
\begin{cases}
a(\boldsymbol{s}, \boldsymbol{r}) 
\coloneqq \frac{24}{25} \operatorname{Kn} (\operatorname{sym} \nabla \boldsymbol{s} ,\operatorname{sym} \nabla\boldsymbol{r}) + \frac{12}{25} \operatorname{Kn} (\nabla\cdot\boldsymbol{s},\nabla\cdot\boldsymbol{r}) 
+ \frac{4}{15} \frac{1}{\operatorname{Kn}} (\boldsymbol{s} , \boldsymbol{r})\\
\qquad\qquad
+ \frac{1}{2} \frac{1}{\tilde{\chi}} \int_{\Gamma} s_{n} r_{n} \, \mathrm{d} l 
+ \frac{12}{25} \tilde{\chi} \sum_{i=1}^{2}\int_{\Gamma} s_{t_i} r_{t_i} \, \mathrm{d} l, \\
c(\boldsymbol{r}, \boldsymbol{\sigma}) 
\coloneqq \frac{2}{5} ({\boldsymbol{\sigma}} , \nabla\boldsymbol{r}) - \frac{3}{20} \int_{\Gamma} \sigma_{nn} r_{n} \, \mathrm{d} l - \frac{1}{5} \sum_{i=1}^{2}\int_{\Gamma} \sigma_{n t_i} r_{t_i} \, \mathrm{d} l, \\
d(\boldsymbol{\sigma}, \boldsymbol{\tau}) 
\coloneqq \operatorname{Kn} \left(\operatorname{Stf} \nabla{\boldsymbol{\sigma}} , \operatorname{Stf} \nabla{\boldsymbol{\tau}}\right)+ \frac{1}{2} \frac{1}{\operatorname{Kn}} ({\boldsymbol{\sigma}} ,
{\boldsymbol{\tau}})+ \frac{9}{8}\tilde{\chi} \int_{\Gamma} \sigma_{nn} \tau_{nn}
\mathrm{d} l \\
\qquad\qquad
+ \tilde{\chi} \int_{\Gamma} \left( \sigma_{t_1 t_1} + \frac{1}{2} \sigma_{nn} \right) \left( \tau_{t_1 t_1} + \frac{1}{2} \tau_{nn} \right) \, \mathrm{d} l+ \tilde{\chi} \int_{\Gamma} \sigma_{t_1 t_2} \tau_{t_1 t_2} \, \mathrm{d} l\\
\qquad\qquad
+ \frac{1}{\tilde{\chi}} \sum_{i=1}^2 \int_{\Gamma} \sigma_{n t_i} \tau_{n t_i} \, \mathrm{d} l,\\
b(\theta, \boldsymbol{r}) 
\coloneqq (\theta,\nabla\cdot \boldsymbol{r}),
\quad
e(\boldsymbol{u}, 
\boldsymbol{\tau}) 
\coloneqq\big(\nabla\cdot\boldsymbol{\tau}, \boldsymbol{u}\big),
\quad 
g(p, \boldsymbol{v}) 
\coloneqq\big(\boldsymbol{v}, \nabla p\big).
\end{cases}
\end{equation}

For $d=2$, invoking a symmetry argument allows the function spaces to be simplified via zero-extension of vectors and the natural embedding of tensors:
\begin{equation}
    \label{eq:auxiliary_tensors}
    \boldsymbol{\sigma}\mapsto
    \tilde{\boldsymbol{\sigma}}\coloneqq\left[\begin{smallmatrix}
        \sigma_{11}&\sigma_{12}&0\\
        \sigma_{12}&\sigma_{22}&0\\
        0&0&-(\sigma_{11}+\sigma_{22})\\
    \end{smallmatrix}\right],\,
    \boldsymbol{\tau}\mapsto
    \tilde{\boldsymbol{\tau}}\coloneqq\left[\begin{smallmatrix}
        \tau_{11}&\tau_{12}&0\\
        \tau_{12}&\tau_{22}&0\\
        0&0&-(\tau_{11}+\tau_{22})\\
    \end{smallmatrix}\right],
\end{equation}
with the corresponding bilinear forms the same as those in \cite[Section 3.1]{theisen2021fenics}.

\section{Penalty-free mixed finite elements scheme}
\label{sec:mfem}

\subsection{Finite element spaces and Galerkin formulation}
Let $\mathcal T_h$ be a conforming triangulation of a domain ${\Omega}\subset \mathbb R^d$ into non-degenerate simplices. 
We assume that the family of triangulations $\{\mathcal T_h\}_{h >0}$ is \textit{shape-regular} and \textit{quasi-uniform} (see \cite{brenner2008mathematical}). 
Additionally, for $d=3$, we impose a mild regularity assumption on the mesh: every tetrahedron must contain at least one vertex in the domain interior to facilitate the analysis.
For each element $K \in\mathcal{T}_h$, 
we denote its diameter by $h_K \coloneqq \operatorname{diam}(K)$
and define the global mesh size as $h \coloneqq \max_{K \in\mathcal{T}_h} h_K$. 

For any simplex $K\in\mathcal T_h$,
we denote by $\Delta_l(K)$ the set of its $l$-dimensional sub-simplices for $0\le l\le d$.
For instance, 
$\Delta_d(K)=\{K\}$ consists of the simplex itself,
while 
$\Delta_0(K)$ is the set of its $d+1$ vertices.

We denote by $\mathcal P_{m}(D;{\mathbb X})$
the space of polynomials of degree at most $m$
on a domain $D$ with values in  $\mathbb X$.
Then the continuous $\mathbb X$-valued Lagrange elements of degree $m$
is defined as
$\mathbb{P}_{h}^{m}(\Omega;\mathbb X)\coloneqq\{v_h \in H^1(\Omega;\mathbb X):v_h|_{K}\in \mathcal P_m(K;\mathbb X)~\forall K \in \mathcal{T}_h\}.$

On each element $K \in \mathcal{T}_h$, we define the volume bubble function $b_K \coloneqq \prod_{i=1}^{d+1} \lambda_i$, where $\lambda_i$ are the barycentric coordinates of $K$. Let $\mathscr{B}_{K}^m \coloneqq \{ b_K \mathcal{E}(p) : p \in \mathcal{P}_{m}(K;\mathbb{R}^d) \}$ denote the local bubble space.
For $m \ge d$, we introduce the finite element quintuplet 
$(\Sigma_h^m, \mathrm{U}_h^m, \mathrm{P}_h^m,\mathrm{S}_h^m, \Theta_h^m)$ defined by:
\begin{align}
\label{eq:FESpaces}
\begin{cases}
\Sigma_h^{m} \coloneqq \left\{\boldsymbol{\tau}_h\in H^1(\Omega;\mathbb T):\boldsymbol{\tau}_h|_{K}\in \big(\mathcal{P}_{m}(K;\mathbb T)\oplus \mathscr B^m_K\big)\,\,\forall K \in \mathcal T_h\right\},\\ 
{\rm U}_h^{m} \coloneqq \mathbb P^m_h(\Omega;\mathbb R^d),\quad
{\rm S}_h^{m} \coloneqq \mathbb P^m_h(\Omega;\mathbb R^d),\\
\mathrm{P}_h^{m} \coloneqq \mathbb P^{m-1}_h(\Omega)\cap \tilde{H}^1(\Omega),
\quad
{\rm \Theta}_h^{m} \coloneqq \mathbb P^{m-1}_h(\Omega).    
\end{cases}
\end{align}

In the subsequent analysis, we focus on the case $m=d$ for dimensions $d=2,3$ and omit the superscript $m$ for brevity.
This spatial discretization pairs a bubble-enriched tensor element with two copies of $\mathbb P_d$-$\mathbb P_{d-1}$ Taylor-Hood elements.
Also, we refer to the quintuplet of spaces $(\Sigma_h,{\rm U}_h,{\rm P}_h,{\rm S}_h,{\Theta}_h)$ as the $\mathbb {P}_d^{\rm b}$-$\mathbb{P}_d$-$\mathbb P_{d-1}$-$\mathbb{P}_d$-$\mathbb{P}_{d-1}$ element.

\begin{remark}[Necessity for $m\geq d$]
The condition $m \ge d$ is sufficient to ensure the existence of a trivial right inverse of the divergence operator in $\Sigma^m_h$ for $\ker(\mathcal E)$. 
As detailed in \cref{lem:characteristic_kernel}, this kernel satisfies the inclusion $\mathcal P_{d-2} \subsetneq \ker(\mathcal E) \subsetneq \mathcal P_{d-1}$.
\end{remark}

\begin{remark}[Why 2D]
While $d=2$ can be viewed as a special case of $d=3$, adopting this perspective would impose the restriction $m \ge 3$ required by the 3D analysis. 
This would preclude the use of the lowest-order element ($m=2$) in two dimensions. Given the significant expense associated with higher-order elements, we focus primarily on the minimal case $m=d$. Nevertheless, the proposed scheme is readily extensible to $m > d$ for higher orders of convergence.
\end{remark}


Let $(\boldsymbol{\tau}_h, \boldsymbol{r}_h, q_h, \boldsymbol{v}_h, \gamma_h)$ denote the test functions corresponding to solution fields $(\boldsymbol{\sigma}_h, \boldsymbol{s}_h, p_h, \boldsymbol{u}_h, \theta_h)$, respectively.
For simplicity, we define
\begin{equation}
\label{eq:grouping_discrete}
\begin{aligned}
&\text{Galerkin Solutions:}&
\boldsymbol{\mathcal{S}}_h \coloneqq (\boldsymbol{\sigma}_h, \boldsymbol{s}_h, p_h) \in \mathrm{T}_h, \quad &\boldsymbol{\mathcal{U}}_h \coloneqq (\boldsymbol{u}_h, \theta_h) \in \mathrm{W}_h, \\
&\text{Discrete Test Functions:}&
\boldsymbol{\mathcal{R}}_h \coloneqq (\boldsymbol{\tau}_h,
\boldsymbol{r}_h, q_h) \in \mathrm{T}_h, \quad &\boldsymbol{\mathcal{V}}_h \coloneqq (\boldsymbol{v}_h, \gamma_h) \in \mathrm{W}_h,
\end{aligned}
\end{equation}
where
${\rm T}_h\coloneqq \Sigma_h \times {\rm U}_h \times {\rm P}_h$,
${\rm W}_h\coloneqq {\rm U}_h \times \Theta_h$.
Then the Galerkin formulation reads: 
find \(\boldsymbol{\mathcal{S}}_h \in \mathrm{T}_h,\,\, \boldsymbol{\mathcal{U}}_h \in \mathrm{W}_h\) such that
\begin{equation}
\begin{aligned}
\label{eq:MFEM_Scheme}
\mathcal{A}(\boldsymbol{\mathcal{S}}_h,\boldsymbol{\mathcal{R}}_h)+ \mathcal{B}(\boldsymbol{\mathcal{R}}_h,\boldsymbol{\mathcal{U}}_h) &= \mathcal{F}(\boldsymbol{\mathcal{R}}_h) \quad &\forall\, \boldsymbol{\mathcal{R}}_h &\in \mathrm{T}_h,
\\
\mathcal{B}(\boldsymbol{\mathcal{S}}_h,\boldsymbol{\mathcal{V}}_h) &= 0\quad &\forall\, \boldsymbol{\mathcal{V}}_h &\in \mathrm{W}_h.
\end{aligned}
\end{equation}

\subsection{Bubble-enriched element space and interpolation operator}
Following \cite{an2024decoupled}, the local degrees of freedom (DoFs) for $\Sigma_h$ on $K$ are defined by
\begin{equation}
    \label{eq:Sigmah_dofs}
    \begin{array}{ll}
        \boldsymbol{\tau}(\delta), & \forall \delta\in\Delta_0(K), \\
        (\boldsymbol{\tau},\boldsymbol{\sigma})_{f}, & \forall\boldsymbol{\sigma}\in \mathcal{P}_{d-(r+1)}(f;\mathbb{T}),\,
        f\in\Delta_r(K),\,
        r=1,\cdots,d-1,\\
(\boldsymbol{\tau},\boldsymbol{\sigma})_{K},&\forall\boldsymbol{\sigma}\in {\stfD} \mathcal P_{d}(K;\mathbb R^d).
    \end{array}
\end{equation}
The unisolvence of \cref{eq:Sigmah_dofs} is established in \cite[Lemma 2.1]{an2024decoupled}.

Notably, these DoFs induce an $H^1$-stable interpolator that preserves homogeneous traces for non-smooth tensor fields.
Define the subspace with homogeneous boundary conditions as
$
\Sigma_{h,0}:= \Sigma_h \cap H^1_0(\Omega;\mathbb T).
$
Let $Q_K:L^2(K;\mathbb{T})\rightarrow \mathcal{P}_{d}(K;\mathbb{T})\oplus \mathscr{B}_{K}$ be the local $L^2$-orthogonal projection operator on $K$. 
The interpolation operator ${\cal I}_h:H_0^1(\Omega;\mathbb{T})\to\Sigma_{h,0}$
is uniquely determined by:
\begin{equation}
    \label{eq:interp_operator}
    \begin{aligned}
        {\cal I}_h\boldsymbol{\tau}(\delta)&:=\frac{1}{\#\mathcal{T}_{\delta}}\sum_{K\in\mathcal{T}_\delta}(Q_K\boldsymbol{\tau})(\delta),\\
        ({\cal I}_h\boldsymbol{\tau},\boldsymbol{\sigma})_{f}&:=\frac{1}{\#\mathcal{T}_{f}}\sum_{K\in\mathcal{T}_f}(Q_K\boldsymbol{\tau},\boldsymbol{\sigma})_{f}\quad\forall\boldsymbol{\sigma}\in\mathcal{P}_{k-(r+1)}(f;\mathbb{T}),\\
        ({\cal I}_h\boldsymbol{\tau},\boldsymbol{\sigma})_{K}&:=(\boldsymbol{\tau},\boldsymbol{\sigma})_{K}\quad
        \forall\boldsymbol{\sigma}\in{\cal E}(\mathcal{P}_d(K;\mathbb{R}^d)),
    \end{aligned}
\end{equation}
for all interior vertices $\delta$, interior $r$-dimensional sub-simplices $f$ ($r=1,\cdots,d-1$) and $d$-dimensional simplex elements $K$ of $\mathcal T_h$.
Here, $\mathcal{T}_{\delta}$ and $\mathcal{T}_f$ denote the sets of elements in $\mathcal{T}_h$ sharing the vertex $\delta$ or the sub-simplex $f$, respectively. Analogous to \cite[Lemma 2.2]{an2024decoupled}, this operator satisfies the following stability and approximation properties:
\begin{lemma}
    \label{Lem:interp_op}
    The interpolation operator ${\cal I}_h:H_0^1(\Omega;\mathbb{T})\rightarrow\Sigma_{h,0}$ defined by \cref{eq:interp_operator} is $H^1$-stable and preserves the homogeneous trace. It has the following properties:
    \begin{itemize}
        \item (Homogeneous trace preservation) ${\cal I}_h\big({H_0^1(\Omega;\mathbb{T})}\big)\subseteq {H_0^1(\Omega;\mathbb{T})}$,
        \item ($H^1$-stability) $\|{\cal I}_h\boldsymbol{\tau}\|_1\lesssim|\boldsymbol{\tau}|_{1}$,
        \item ($L^2$-approximation property) $\|\boldsymbol{\tau}-{\cal I}_h\boldsymbol{\tau}\|_{0}\lesssim h|\boldsymbol{\tau}|_{1}$,
        \item ($H^1$-approximation property) $|\boldsymbol{\tau}-{\cal I}_h\boldsymbol{\tau}|_1\lesssim|\boldsymbol{\tau}|_1$.
    \end{itemize}
\end{lemma}
\begin{remark}
    For the case $d=2$, \cref{Lem:interp_op} corresponds to the result established in \cite[Lemma 2.2]{an2024decoupled}; a detailed proof is also provided in our supplementary material. 
    For $d=3$, the result follows from analogous arguments, 
    noting that the space of symmetric trace-free tensors is a subspace of the space of symmetric tensors.
\end{remark}

\subsection{Kernel space and orthogonal decomposition}
To establish a unified framework for both two- and three-dimensional cases,
we begin by characterizing the kernel of $\stfD$ and deriving an orthogonal decomposition of the discrete space ${\rm U}_h$.

Denote by $\cal K$ the kernel of $\stfD$
in $L^2(\Omega;\mathbb R^d)$ in the distributional sense:
\begin{equation}
    \mathcal{K} \coloneqq \ker(\stfD) = \{ \boldsymbol{v} \in L^2(\Omega; \mathbb{R}^d) \mid \stfD \boldsymbol{v} = \boldsymbol{0} \}.
\end{equation}

We now present a series of lemmas characterizing the properties of $\stfD$ and $\mathcal{K}$.

\begin{lemma}[Adjoint Operator] \label{lem:adjoint_div_grad}
The operator $\stfD: L^2(\Omega; \mathbb{R}^d) \to H^{-1}(\Omega; \mathbb{T})$ is the adjoint of divergence operator $-\nabla\cdot: H_0^1(\Omega; \mathbb{T}) \to L^2(\Omega; \mathbb{R}^d)$, i.e., 
$-{\nabla \cdot} = (\stfD)'$.
\end{lemma}

\begin{proof}
Let $\boldsymbol{v} \in L^2(\Omega; \mathbb{R}^d)$ and $\boldsymbol{\tau} \in H_0^1(\Omega; \mathbb{T})$. 
A key observation is that the gradient $\nabla \boldsymbol{v}$ can be decomposed as
$
\nabla\boldsymbol{v} = \stfD \boldsymbol{v} + (\nabla \boldsymbol{v} - \stfD \boldsymbol{v}).
$

By the definitions of $\stfD$ and $\mathbb{T}$, the second term $(\nabla\boldsymbol{v} - \stfD \boldsymbol{v})$ is orthogonal to any tensor in $\mathbb{T}$ with respect to the Frobenius inner product. 
For $d=2$, $\nabla\boldsymbol{v} - \stfD\boldsymbol{v} = \mathrm{skew}(\nabla\boldsymbol{v})$, which is orthogonal to the symmetric tensor $\boldsymbol{\tau}$. 
For $d=3$, $\nabla\boldsymbol{v} - \stfD\boldsymbol{v}$ is a sum of a skew-symmetric and a pure trace tensor, both of which are orthogonal to the symmetric and trace-free tensor $\boldsymbol{\tau}$.
Thus, in both cases, we have $(\stfD \boldsymbol{v}, \boldsymbol{\tau}) = (\nabla\boldsymbol{v}, \boldsymbol{\tau})$. 


Using this identity and integration by parts, we obtain
\begin{equation}
    \langle \stfD \boldsymbol{v}, \boldsymbol{\tau} \rangle_{H^{-1}, H_0^1} = (\stfD \boldsymbol{v}, \boldsymbol{\tau}) = (\nabla\boldsymbol{v}, \boldsymbol{\tau}) = -(\boldsymbol{v}, \nabla\cdot\,\boldsymbol{\tau}), 
\end{equation}
which completes the proof.
\end{proof}

Next, we characterize the kernel of the operator ${\cal E}$.
\begin{lemma}[Characterization of the Kernel]
\label{lem:characteristic_kernel}
The kernel of $\mathcal{E}$ in the distributional sense, ${\cal K} \coloneqq \ker(\mathcal{E}|_{L^2})$, is a finite-dimensional space of polynomials. Specifically:
\begin{itemize}
    \item For $d=2$, $\mathcal{K}$ is the space of rigid motions $\bf RM$, which are affine maps of the form $\boldsymbol{v}(\mathbf{x}) = \mathbf{a} + A\mathbf{x}$ where $\mathbf{a} \in \mathbb{R}^2$ and $A \in \mathbb{R}^{2\times 2}$ is a skew-symmetric matrix. These are polynomials of degree at most one, leading to ${\cal P}_0 \subsetneq \mathbf{RM} \subsetneq \mathcal{P}_1$. 
  
    \item For $d=3$, $\mathcal{K}$ is the space of conformal Killing maps $\bf CK$ composed of
    \begin{equation} \label{eq:conformal_killing_map}
        \boldsymbol{v}(\mathbf{x}) = \mathbf{a} + \lambda\mathbf{x} + A\mathbf{x} + (2(\mathbf{b}\cdot\mathbf{x})\mathbf{x} - \|\mathbf{x}\|^2\mathbf{b}),
    \end{equation}
    for some constants $\mathbf{a}, \mathbf{b} \in \mathbb{R}^3$, $\lambda \in \mathbb{R}$, and a skew-symmetric matrix $A \in \mathbb{R}^{3\times 3}$. These are polynomials of degree at most two, leading to ${\cal P}_1 \subsetneq \mathbf{CK} \subsetneq \mathcal{P}_2$.
\end{itemize}
\end{lemma}
\begin{proof}
For a rigorous proof, see, e.g., \cite[Theorem 2.2]{ciarlet2010onKorn} for $\mathbf{RM}$
and \cite[Proposition 2.5]{schirra2012new} for $\mathbf{CK}$. 
\end{proof}

Let $\mathrm{Y} \subset L^2(\Omega;\mathbb{R}^d)$ be a finite-dimensional space, for instance, $\mathbb P^k_h(\Omega;\mathbb R^d)$. 
The next lemma shows that if $\mathrm{Y}$ is sufficiently rich, 
its kernel under $\stfD$ is precisely $\mathcal{K}$.
\begin{lemma}[Invariance of the Kernel] \label{lem:kernel_invariance}
Let $\mathrm{Y}$ be a finite-dimensional subspace of $L^2(\Omega; \mathbb{R}^d)$ such that $\mathcal{K} \subseteq {\rm Y}$. Then the kernel of $\stfD$ restricted to $\rm{Y}$ coincides with $\mathcal{K}$.
\end{lemma}

\begin{proof}
By definition, $\mathcal{K} = \ker(\stfD|_{L^2})$. The inclusion $\mathcal{K} \subseteq \ker(\stfD|_{\rm Y})$ is immediate from the assumption $\mathcal{K} \subseteq {\rm Y}$. The converse direction is also obvious.
\end{proof}

This result allows for a unique orthogonal decomposition of both the continuous space $L^2(\Omega; \mathbb{R}^d)$ and the discrete space $\mathrm{U}_h = \mathbb{P}^d_h(\Omega;\mathbb{R}^d)$ with respect to the kernel $\mathcal{K}$. 
Let $\mathcal{K}^\perp$ be the orthogonal complement of $\mathcal{K}$ in $L^2(\Omega; \mathbb{R}^d)$. 
By taking ${\rm Y}={\rm U}_h$, we have the following $L^2$-orthogonal decompositions:
\begin{align}
\label{eq:L2_orthgonal_decomposition}
    L^2(\Omega; \mathbb{R}^d) = \mathcal{K} \oplus \mathcal{K}^\perp, \quad
    \mathrm{U}_h = \mathcal{K} \oplus \mathrm{U}_{h,0},
\end{align}
where the orthogonal complement of $\mathcal{K}$ within the subspace $\mathrm{U}_h$ is denoted as ${\mathrm{U}_{h,0}}$,
\begin{equation}
\mathrm{U}_{h,0} \coloneqq \mathcal{K}^\perp \cap \mathrm{U}_h.    
\end{equation}

\begin{lemma}
\label{lem:subspace_property}
The space $\mathrm{U}_{h,0}$ is a closed subspace of $\mathcal{K}^\perp$.
\end{lemma}

\begin{proof}
By definition, $\mathrm{U}_{h,0} = \mathcal{K}^\perp \cap \mathrm{U}_h$. Since $\mathcal{K}$ is a subset of the Hilbert space $L^2(\Omega; \mathbb{R}^d)$, its orthogonal complement $\mathcal{K}^\perp$ is a closed subspace. As a finite-dimensional subspace of a normed space, $\mathrm{U}_h$ is closed. Hence $\mathrm{U}_{h,0}$ is a closed subspace in $\mathcal{K}^\perp$.
\end{proof}

Finally, we point out that ${\cal I}_h$ is $B$-compatible for ${\rm U}_{h,0}$.
\begin{lemma}[$B$-Compatibility] 
$\forall \boldsymbol{u}_h\in \mathrm{U}_{h,0}$, $\left(\nabla\cdot(\boldsymbol{\tau}-{\cal I}_h\boldsymbol{\tau}),\boldsymbol{u}_h\right)=0$.    
\end{lemma}
\begin{proof}
Since $\boldsymbol{\tau}-{\cal I}_h\boldsymbol{\tau} \in H^1_0(\Omega;\mathbb{T})$,
\cref{lem:adjoint_div_grad} and the definition of $\mathcal I_h$
yield
\begin{equation}
\left(\nabla\cdot(\boldsymbol{\tau}-{\cal I}_h\boldsymbol{\tau}),\boldsymbol{u}_h\right)_K=
-\left(\boldsymbol{\tau}-{\cal I}_h\boldsymbol{\tau},\stfD\boldsymbol{u}_h\right)_K=0
\end{equation}
for each element $K \in \mathcal T_h$.
\end{proof}

\section{Stability and convergence}
\label{sec:stability_and_convergence}
In this section, we establish the stability of the penalty-free MFEM scheme and the order of convergence under certain mild regularity assumptions.
The unisolvence of MFEM is established via coercivity on the kernel and the inf-sup condition:
\begin{enumerate}
\item \textit{Coercivity on discrete kernel.} 
Let $\ker B_h \coloneqq \{\boldsymbol{\mathcal{S}}_h \in \mathrm{T}_h :\, \mathcal{B}(\boldsymbol{\mathcal{S}}_h, \boldsymbol{\mathcal{V}}_h)=0,\, \forall \,\boldsymbol{\mathcal{V}}_h \in \mathrm{W}_h\}$. 
Coercivity requires $\mathcal{A}(\boldsymbol{\mathcal{S}}_h, \boldsymbol{\mathcal{S}}_h) \gtrsim \|\boldsymbol{\mathcal{S}}_h\|^2$ for all $\boldsymbol{\mathcal{S}}_h \in \ker B_h$. 
By derivation, it suffices to show that 
for any $(\boldsymbol{\sigma}_h,\boldsymbol{s}_h,p_h)\in \ker B_h$,
\begin{equation}
\label{eq:SufficientCondition_coercive}
\tag{U-1}
 \begin{aligned}
      \left\|{\rm sym~\nabla} \boldsymbol{s}_h\right\|_0^2
      + \|\boldsymbol{s}_h\|_0^2 
      + \left\|{\rm Stf~\nabla} \boldsymbol{\sigma}_h\right\|_{0}^2 + \|\boldsymbol{\sigma}_h\|_0^2 \gtrsim \|\boldsymbol{s}_h\|_1^2 + 
    \|\boldsymbol{\sigma}_h\|_1^2 +
    \|p_h\|_1^2.
 \end{aligned}
\end{equation}

\item \textit{Discrete inf-sup condition.}
The discrete stability condition for $({\rm T}_h,{\rm W}_h)$,
\begin{equation}
\inf_{\boldsymbol{\mathcal{V}}_h \in \mathrm{W}_h} \sup_{\boldsymbol{\mathcal{S}}_h \in \mathrm{T}_h}
\frac{-(\boldsymbol{v}_h, \nabla \cdot \boldsymbol{\sigma}_h) - (\nabla p_h, \boldsymbol{v}_h) - (\gamma_h, \nabla \cdot \boldsymbol{s}_h)}
{\left(\|\boldsymbol{\sigma}_h\|^2_{1} + \|\boldsymbol{s}_h\|_1^2 + \|p_h\|_{1}^2\right)^{1/2} \left(\|\boldsymbol{v}_h\|_0^2 + \|\gamma_h\|_0^2\right)^{1/2}} \gtrsim 1
\end{equation}
is satisfied provided the decoupled conditions hold for all $\boldsymbol{v}_h\in {\rm U_h}$, $ \gamma_h \in \Theta_h$:
\begin{equation}
\label{eq:SufficientCondition_infsup}
\tag{U-2}
\begin{aligned}
\sup_{\boldsymbol{\sigma}_h\in \Sigma_h }\frac
{-
(\boldsymbol{v}_h,\nabla \cdot\boldsymbol{\sigma}_h)}{\|\boldsymbol{\sigma}_h\|_{1}}
\gtrsim \|\boldsymbol{v}_h\|_0,
\qquad 
\sup_{\boldsymbol{s}_h\in {\rm S}_h}\frac
{-
(\gamma_h, \nabla \cdot \boldsymbol{s}_h)}{
\|\boldsymbol{s}_h\|_1}
\gtrsim  \|\gamma_h\|_0.
\end{aligned}
\end{equation}

These decoupled conditions ensure the existence of non-trivial $(\boldsymbol{\sigma}_h, \boldsymbol{s}_h) \in \Sigma_h \times \mathrm{S}_h$ satisfying the following stability estimates with constants $C_1, C_2 > 0$:
$$
-
(\boldsymbol{v}_h,\nabla \cdot{\boldsymbol{\sigma}_h})
\geq
C_1\|\boldsymbol{v}_h\|_0\|{\boldsymbol{\sigma}_h}\|_1,
\quad
-
(\gamma_h, \nabla \cdot{\boldsymbol{s}_h})
\geq
C_2\|\gamma_h\|_0\|{\boldsymbol{s}}_h\|_1.$$
Then constructing ${\cal{S}}_h = (\boldsymbol{\sigma}_h,\boldsymbol{s}_h,0) \in {\rm T}_h$
and scaling the components such that $\|\boldsymbol{\sigma}_h\|_{1}=C_1\|\boldsymbol{v}_{h}\|_0$ and $\|\boldsymbol{s}_h\|_1=C_2\|\gamma_h\|_{0}$
yield
\begin{equation}
\begin{aligned}
\sup_{\boldsymbol{\mathcal{S}_h}\in \mathrm{T}_h} \frac{\mathcal B(\boldsymbol{\mathcal{S}}_h,\boldsymbol{\mathcal V}_h)}{\|\boldsymbol{\mathcal{S}}_h\|_{\mathrm{T}}}
\geq
{\sqrt{\|{\boldsymbol{\sigma}}_h\|^2_{1}+
\|{\boldsymbol{s}}_h\|_1^2}} 
\gtrsim 
{\sqrt{\|\boldsymbol{v}_h\|_0^2 + \|\gamma_h\|_0^2}}
=\|\boldsymbol{\mathcal{V}}_h\|_{\mathrm{W}},
\end{aligned}
\end{equation}
thereby verifying the joint discrete inf-sup condition.  
\end{enumerate}

\subsection{Discrete inf-sup condition}
For the relatively simple pair $(\mathrm{S}_h, \Theta_h)$ from decoupled inf-sup conditions \cref{eq:SufficientCondition_infsup},
the stability of this $\mathbb P_d$-$\mathbb P_{d-1}$ pair
relies on the discrete inf-sup condition 
for the Stokes equations with pure Neumann boundary conditions.
This is usually viewed as a natural corollary of the Dirichlet case. 
However, as we have not located an explicit reference for our specific formulation, we briefly sketch the proof, employing Verfürth's argument.
\begin{theorem}
\label{thm:infsup_vector}
The finite element spaces $(\mathrm{S}_h, \Theta_h)$ satisfy the discrete inf-sup condition,
\begin{equation}
\sup_{\boldsymbol{s}_h\in \mathrm{S}_h}\frac{{
(\gamma_h, \nabla \cdot \boldsymbol{s}_h)}}{\|\boldsymbol{s}_h\|_1}\gtrsim \|\gamma_h\|_0,\quad\forall \gamma_h \in {\Theta}_h.
\end{equation}
\end{theorem}
\begin{proof}
Decompose the temperature space according to $\Theta_h = \Theta_{h,0} \oplus \mathbb R$ such that $\Theta_{h,0}$ is zero-mean space. 
For any $\gamma_h \in \Theta_h$, we write $\gamma_h = \gamma_{h,0} + \hat \gamma$ correspondingly.

We first show that for $\gamma_h \in \Theta_h$, there exists $\boldsymbol{r}_h\in\rm S_h$ 
with $\|\boldsymbol{r}_h\|_1 \lesssim \|\gamma_h\|_0$
and 
\begin{equation} 
\label{eq:infsup_verfurth1} 
\left(\gamma_h, \nabla \cdot \boldsymbol{r}_h\right) \gtrsim \left\|\gamma_h\right\|_0^2 - h\left\|\gamma_h\right\|_0 \left\|\nabla \gamma_h\right\|_0\gtrsim 
\left\|\boldsymbol{r}_h\right\|_1 
\left(\left\|\gamma_h\right\|_0 - h\left\|\nabla\gamma_h\right\|_0\right).
\end{equation}

For the zero-mean component $\gamma_{h,0}$,
there exists $\boldsymbol{r}_0 \in H^1_0(\Omega;\mathbb R^d)$ such that $\nabla \cdot \boldsymbol{r}_0 = \gamma_{h,0}$ with $\|\boldsymbol{r}_0\|_1 \lesssim \|\gamma_{h,0}\|_0$,
which follows from the existence of a right inverse for divergence (see, e.g., \cite[Theorem IV.3.1]{boyer2012mathematical}).
Then denote by $\boldsymbol{r}_{h,0}$ the Scott-Zhang interpolation of $\boldsymbol{r}_0$ into ${\rm S}_h \cap H^1_0$ (see \cite{scott1990finite}).
And for the constant term, let $\hat{\boldsymbol{r}}=\hat{\gamma}{\bf x}/d$ with $\nabla\cdot \hat{\boldsymbol{r}} = \hat\gamma$.
Therefore, picking $\boldsymbol{r}_h := \boldsymbol{r}_{h,0} + \hat{\boldsymbol{r}}$ and applying the approximation property of the Scott-Zhang interpolation 
$\|\boldsymbol{r}_0 - \boldsymbol{r}_{h,0}\|_0\lesssim h\|\mathbf{r}_{0}\|_1$
yield \cref{eq:infsup_verfurth1}.

By the same arguments as given in \cite{brezzi1991stability, stenberg1987on3D}, the stability of Taylor-Hood elements under Dirichlet boundary conditions ensures there exists $\boldsymbol{s}_{h,0}\in (\mathrm{S}_h\cap H^1_0)$ such that 
\begin{equation} 
\label{eq:infsup_verfurth2} 
\left(\gamma_{h},\nabla\cdot\boldsymbol{s}_{h,0}\right) \gtrsim h\left\|\boldsymbol{s}_{h,0}\right\|_1\left\|\nabla\gamma_{h}\right\|_{0}.
\end{equation}
Combining \cref{eq:infsup_verfurth1} and \cref{eq:infsup_verfurth2} completes the proof.
\end{proof}

As for the $\mathbb{P}_d^b$-$\mathbb P_{d}$ pair $(\Sigma_h,{\rm U}_h)$, previous numerical experiments (e.g., \cite[Section 3]{westerkamp2019finite}) showed that the standard Taylor-Hood element pair doesn't work well and some penalty terms are needed to ensure the numerical stability. Here, our approach is to enrich Lagrange element spaces with volume bubble functions as a conforming stabilization.
The main issue is then to verify for any $\boldsymbol{u}_{h} \in {\rm U}_h$,
\begin{equation*}
   \sup_{\boldsymbol{\sigma}_h \in \Sigma_h}\frac{(\nabla\cdot \boldsymbol{\sigma}_h,\boldsymbol{u}_{h})}{\|\boldsymbol{\sigma}_h\|_{1}} \gtrsim \|\boldsymbol{u}_{h}\|_{0}.
\end{equation*}

While one might ideally define $\boldsymbol{\sigma}_h = (\nabla\cdot)^{-1}(\boldsymbol{u}_h)$ while preserving $H^1$ stability, such a construction is difficult to realize in the discrete setting. 
We therefore pursue an alternative strategy, inspired by the proof of \cref{thm:infsup_vector} and centered on orthogonal decompositions \cref{eq:L2_orthgonal_decomposition}. 
Our first step involves demonstrating the existence of a right inverse for the divergence operator, restricted to the orthogonal complement of its adjoint's kernel $\mathcal K^{\perp}$.

\begin{theorem}[Existence of right inverse for the divergence operator]
\label{thm:inv_divergence_new_proof}
For $\boldsymbol{u}\in \mathcal{K}^{\perp}$,
there exists $\boldsymbol{\tau} \in H^1_0(\Omega;{\mathbb T})$ such that
\begin{equation}
    \label{eq:inv_divergence_for_Vh_new_proof}
    -\nabla\cdot\boldsymbol{\tau}=\boldsymbol{u},\quad \|\boldsymbol{\tau}\|_{1}\lesssim\|\boldsymbol{u}\|_0.
\end{equation}

\end{theorem}
\begin{proof}
For the $d=2$ case, we show the divergence operator is a surjective map from $H^1_0(\Omega; {\mathbb{R}^{2 \times 2}_{\rm sym}})$ to $\mathbf{RM}^{\perp}$.
By \cref{lem:adjoint_div_grad},
the adjoint of $-\operatorname{div}$ is ${\rm sym~\nabla}:L^2(\Omega;\mathbb R^2)\to H^{-1}(\Omega;{\mathbb{R}^{2 \times 2}_{\rm sym}})$,
and the kernel of the adjoint is the space of rigid motions, ${\bf RM}$.
According to the closed range theorem,
if the range of the adjoint operator ${\rm sym~\nabla}$ is closed in $H^{-1}(\Omega;\mathbb{R}^{2 \times 2}_{\rm sym})$,
then the range of the primal operator $-\operatorname{div}$ equals the orthogonal complement of the adjoint's kernel, that is
\begin{equation}
\operatorname{Range}(\operatorname{div}) = (\ker({\rm sym~\nabla}))^\perp = {\bf RM}^{\perp}.
\end{equation}
To show $\operatorname{Range}({{\rm sym~\nabla}})$ is closed, it suffices to verify that ${\rm sym~\nabla}$ is bounded below on the orthogonal complement of its kernel, that is, for all $\boldsymbol{v} \in \mathbf{RM}^{\perp} \subseteq L^2(\Omega;\mathbb R^2)$,
\begin{equation} \label{eq:korn_neg}
\|{\rm sym~\nabla} \boldsymbol{v}\|_{-1} \gtrsim \|\boldsymbol{v}\|_0.
\end{equation}
The above inequality follows from Korn's inequality of negative norm (see e.g., \cite[Theorem 3.2]{duvant2012inequalities}), which states:
\begin{align}
\label{eq:korn_neg_origin_2D}
\|\boldsymbol{v}\|_0 \lesssim \|{\rm sym~\nabla}\boldsymbol{v}\|_{-1} + \|\boldsymbol{v}\|_{-1}, \quad \forall \boldsymbol{v} \in L^2(\Omega;\mathbb R^2).
\end{align}
Assume for the sake of contradiction that \cref{eq:korn_neg} fails. Then there exists a sequence $\{\boldsymbol{v}_k\} \subset \mathbf{RM}^{\perp}$ with $\|\boldsymbol{v}_k\|_0 = 1$ such that $\|{\rm sym~\nabla} \boldsymbol{v}_k\|_{-1} \to 0$.
Since the embedding $L^2(\Omega) \hookrightarrow H^{-1}(\Omega)$ is compact on bounded Lipschitz domain, there exists a subsequence convergent under $H^{-1}$ norm, which means $$\|\boldsymbol{v}_{k_n}-\boldsymbol{v}_{k_m}\|_{-1}\to 0$$ 
as $n,m\to\infty$. 
By the Korn's inequality \cref{eq:korn_neg_origin_2D} applied to $\boldsymbol{v}_{k_n}-\boldsymbol{v}_{k_m}$, such subsequence is also Cauchy sequence under $L^2$ norm. 
Then let $\boldsymbol{v}^*$ be the limit under the $L^2$ norm.
We have $\|\boldsymbol{v}^*\|_0=1$, $\boldsymbol{v}^* \in \mathbf{RM}^\perp$, and continuity implies ${\rm sym~\nabla}\boldsymbol{v}^* = 0$. Thus $\boldsymbol{v}^* \in \mathbf{RM}$, which implies $\boldsymbol{v}^* = 0$, a contradiction.

Hence, the range of ${\rm sym~\nabla}$ is closed and the operator $-{\rm div}: H_0^1(\Omega;{\mathbb{R}^{2 \times 2}_{\rm sym}}) \to \mathbf{RM}^{\perp}$ is surjective. Since the open mapping theorem guarantees the existence of a continuous right inverse,
for any $\boldsymbol{u} \in \mathbf{RM}^{\perp}$, consequently there exists $\boldsymbol{\tau} \in H_0^1(\Omega;{\mathbb{R}^{2 \times 2}_{\rm sym}})$ satisfying $-\nabla \cdot \boldsymbol{\tau} = \boldsymbol{u}$ such that
$
\|\boldsymbol{\tau}\|_{1} \lesssim \|\boldsymbol{u}\|_0.
$

The proof for the $d=3$ case follows an identical logic to that for $d=2$ with $\bf RM$ replaced by $\bf CK$ and relies on a generalized Korn's inequality of negative norm,
\begin{align*}
\|\boldsymbol{v}\|_0 \lesssim \|{\rm stf~\nabla}\boldsymbol{v}\|_{-1} + \|\boldsymbol{v}\|_{-1}.    
\end{align*}
A detailed proof for the above inequality is provided in \cref{thm:Korn_inequality_stfD_negative_norm}.
\end{proof}

Therefore, combining 
\cref{lem:subspace_property,thm:inv_divergence_new_proof}
yields: 
\begin{theorem}
\label{thm:inf_sup_for_H10_and_U0}
For $\boldsymbol{u}_h \in {\rm U}_{h,0}$,
there exists $\boldsymbol{\tau}_h\in H_0^1(\Omega;\mathbb{T})$ such that 
\begin{equation}
-\nabla\cdot\boldsymbol{\tau}_h=\boldsymbol{u}_h,\quad \|\boldsymbol{\tau}_h\|_{1}\lesssim\|\boldsymbol{u}_h\|_0.
\end{equation}
\end{theorem}

Now we turn to our case and derive the discrete inf-sup condition for the $\mathbb{P}_d^b$-$\mathbb P_{d}$ pair $({\Sigma}_h,{\rm U}_h)$. The proof is also based on a useful lemma related to the trivial right inverse of divergence in $\Sigma_h$ acting on $\mathcal K$, another component of ${\rm U}_h$.

\begin{lemma}
\label{lem:inv_div_for_mathcal_K}
    For $\hat{\boldsymbol{u}}\in\mathcal K$, there exists $\hat{\boldsymbol{\sigma}} \in \Sigma_{h}$
such that $\hat{\boldsymbol{u}} = \nabla \cdot {\hat{\boldsymbol{\sigma}}}$, $\|\hat{\boldsymbol{\sigma}}\|_{1} \lesssim \|\hat{\boldsymbol{u}}\|_{0}$.
\end{lemma}
\begin{proof}
For $d=2$, 
a rigid body motion $\hat{\boldsymbol{u}}$ 
writes as:
$$
\hat{\boldsymbol{u}} 
=\left(
ax_2+b_1,-ax_1+b_2
\right)
= a(x_2 \boldsymbol{e}_1 - x_1 \boldsymbol{e}_2) + b_1 \boldsymbol{e}_1 + b_2 \boldsymbol{e}_2.
$$
Construct 
$
\hat{\boldsymbol{\sigma}} = 
(ax_1x_2 + b_1x_1) \boldsymbol{e}_1 \otimes \boldsymbol{e}_1 
+ (-ax_1x_2 + b_2x_2) \boldsymbol{e}_2 \otimes \boldsymbol{e}_2
\in \mathcal P_2 \subseteq \Sigma_h$.
A direct calculation of the divergence gives
$\nabla \cdot \hat{\boldsymbol{\sigma}} = \hat{\boldsymbol{u}}$. 

Also, note that $\|\hat{\boldsymbol u}\|_0$ can be viewed as a norm on $\mathbb R^3$, let's say,
$\|\cdot\|_{\hat{\boldsymbol u}}:\mathbb R^3\to\mathbb R$ induced by
$\|(a,b,c)\|_{\hat{\boldsymbol u}}:= \|\hat{\boldsymbol u}\|_0
=\left\|
\left(
ax_2+b,-ax_1+c
\right)
\right\|_0$.
We claim $\|\cdot\|_{\hat{\boldsymbol u}}$ is a norm.
The homogeneity and positive definiteness are obvious.
For triangle inequality,
\begin{align*}
\left\|
(a_1,b_1,c_1)+(a_2,b_2,c_2)
\right\|_{\hat{\boldsymbol u}}
&=
\left\|
\hat{\boldsymbol u}_1+\hat{\boldsymbol u}_2
\right\|_0 \\
&
\leq \|\hat{\boldsymbol u}_1\|_0+\|\hat{\boldsymbol u}_2\|_0
= \left\|
(a_1,b_1,c_1)
\right\|_{\hat{\boldsymbol u}}
+
\left\|
(a_2,b_2,c_2)
\right\|_{\hat{\boldsymbol u}}.
\end{align*}
Hence $\|\cdot\|_{\hat{\boldsymbol u}}$ is a norm on $\mathbb R^3$.
Similarly, $\|\hat{\boldsymbol{\sigma}}\|_1$ can induce
$\|(a,b,c)\|_{\hat{\boldsymbol \sigma}}:= \|\hat{\boldsymbol \sigma}\|_1$
as a norm on $\mathbb R^3$ and applying the equivalence of norms in $\mathbb R^3$ completes the proof.   

For $d=3$, 
since $\mathbf{CK} \subseteq \mathcal{P}_2(\Omega;\mathbb{R}^3)$ is a finite-dimensional space of polynomials, for any $\hat{\boldsymbol{u}} \in \mathbf{CK}$, 
it suffices to construct a polynomial tensor $\hat{\boldsymbol{\sigma}} \in \mathcal{P}_3(\Omega;\mathbb{R}^{3\times 3}_{\rm stf})$ such that $\nabla \cdot \hat{\boldsymbol{\sigma}} = \hat{\boldsymbol{u}}$.
Note that any $\hat{\boldsymbol{u}}\in{\bf CK}$ can be decomposed into four different parts:
\begin{align}
\hat{\mathbf{u}}(\mathbf{x}) = \underbrace{\mathbf{a}}_{\text{Translation}} + \underbrace{A\mathbf{x}}_{\text{Rotation}} + \underbrace{\lambda \mathbf{x}}_{\text{Scaling}} + \underbrace{(2(\mathbf{b}\cdot\mathbf{x})\mathbf{x} - \|\mathbf{x}\|^2 \mathbf{b})}_{\text{Special Conformal}}.
\end{align}
We construct $\hat{\boldsymbol{\sigma}}\in\mathcal{P}_3(\Omega; \mathbb{R}^{3\times 3}_{\rm stf})\subseteq \Sigma_h$ for each type ${\hat{\boldsymbol{u}}}$ such that $\nabla\cdot\hat{\boldsymbol{\sigma}}=\hat{\boldsymbol{u}}$ as follows:
\begin{itemize}
\vspace{1ex}
\item For translation motion $\hat{\boldsymbol{u}} = \mathbf{a}$, 
define 
$
\hat{\boldsymbol{\sigma}} = \big(3(\mathbf{a}\otimes\mathbf{x} + \mathbf{x}\otimes\mathbf{a}) - 2(\mathbf{a}\cdot\mathbf{x})\boldsymbol{I}\big)/10.
$
\vspace{1ex}
\item For rotation motion $\hat{\boldsymbol{u}} = A\mathbf{x}$, 
define 
$
\hat{\boldsymbol{\sigma}} = \left((A\mathbf{x})\otimes\mathbf{x} + \mathbf{x}\otimes(A\mathbf{x})\right)/5.
$
\vspace{1ex}
\item  For scaling motion $\hat{\boldsymbol{u}} = \lambda\mathbf{x}$, define 
$
\hat{\boldsymbol{\sigma}} = 3\lambda\big(\mathbf{x}\otimes\mathbf{x} - \frac{1}{3}\|\mathbf{x}\|^2\boldsymbol{I}\big)/10.
$
\vspace{1ex}
\item For special conformal motion $\hat{\boldsymbol{u}} = 2(\mathbf{b}\cdot\mathbf{x})\mathbf{x} - \|\mathbf{x}\|^2\mathbf{b}$, define
$$
\hat{\boldsymbol{\sigma}} = \frac{1}{70}\left( 34(\mathbf{b}\cdot\mathbf{x})\mathbf{x}\otimes\mathbf{x} - 11\|\mathbf{x}\|^2(\mathbf{b}\otimes\mathbf{x} + \mathbf{x}\otimes\mathbf{b}) - 4\|\mathbf{x}\|^2(\mathbf{b}\cdot\mathbf{x})\boldsymbol{I} \right).
$$
\end{itemize}

Then applying the equivalence of norms in finite-dimensional spaces, as in the proof of $d=2$, completes the proof.
\end{proof}

We come to the discrete inf-sup condition for $(\Sigma_h,\mathrm{U}_{h})$
at the end of this subsection:
\begin{theorem}
\label{thm:inf_sup_for_Sigma_and_U}
    The finite element spaces $(\Sigma_h,\mathrm{U}_{h})$ satisfy the discrete inf-sup condition,
    \begin{equation}
        \label{eq:inf_sup_for_Sigma_and_U}
        \inf_{\boldsymbol{u}_h\in {\rm U_h}} \sup_{\boldsymbol{\sigma}_h \in \Sigma_h}\frac{(\boldsymbol{u}_h,\nabla\cdot\boldsymbol{\sigma}_h)}{\|\boldsymbol{u}_h\|_0\|\boldsymbol{\sigma}_h\|_1}\gtrsim 1.
    \end{equation}
\end{theorem}

\begin{proof}
Any given $\boldsymbol{u}_{h}\in {\rm U}_{h}$ admits the decomposition:
\begin{equation}
\label{eq:split_uh_3D}   \boldsymbol{u}_{h}=\boldsymbol{u}_{h,0}+\hat{\boldsymbol{u}},\quad \boldsymbol{u}_{h,0}\in {\rm U_{h,0}},\quad\hat{{\boldsymbol{u}}}\in\mathcal{K}.
\end{equation}
By \cref{thm:inf_sup_for_H10_and_U0}, there exists ${\boldsymbol{\sigma}}_0 \in H^1_0(\Omega;\mathbb T)$ such that,
\begin{equation}
\nabla\cdot {\boldsymbol{\sigma}}_0 = \boldsymbol{u}_{h,0},\quad \|{\boldsymbol{\sigma}}_0\|_1 \lesssim \|\boldsymbol{u}_{h,0}\|_0.
\end{equation}
Denote ${\boldsymbol{\sigma}}_{h,0}:={\cal I}_{h}{\boldsymbol{\sigma}}_0\in{\Sigma}_{h,0}$ where the interpolation operator ${\cal I}_h$ is defined in \cref{eq:interp_operator}.
Also, by \cref{lem:inv_div_for_mathcal_K}, there exists $\hat{{\boldsymbol{\sigma}}}\in{\Sigma}_{h}$ such that,
\begin{equation}
\label{eq:hat_u_3D}
\nabla\cdot\hat{{\boldsymbol{\sigma}}}=\hat{{\boldsymbol{u}}},\quad\|\hat{{\boldsymbol{\sigma}}}\|_{1}\lesssim\|\hat{{\boldsymbol{u}}}\|_{0}.
\end{equation}

Construct $\boldsymbol{\sigma}_{h} \coloneqq \boldsymbol{\sigma}_{h,0}+\hat{\boldsymbol{\sigma}}$ and then we have:
\begin{equation}
\label{eq:split_I_3D}
(\boldsymbol{u}_h,\nabla\cdot\boldsymbol{\sigma}_h) =  \underbrace{(\boldsymbol{u}_{h,0}, \nabla \cdot \boldsymbol{\sigma}_{h,0})}_{I_1}
+ \underbrace{(\hat{\boldsymbol{u}}, \nabla \cdot \boldsymbol{\sigma}_{h,0})}_{I_2}
+ \underbrace{(\boldsymbol{u}_{h,0}, \nabla \cdot \hat{\boldsymbol{\sigma}})}_{I_3}
+ \underbrace{(\hat{\boldsymbol{u}}, \nabla \cdot \hat{\boldsymbol{\sigma}})}_{I_4}.
\end{equation}

For the first term, the $B$-compatibility of ${\cal I}_h$ yields $I_{1} = (\boldsymbol{u}_{h,0}, \boldsymbol{u}_{h,0}) = \|\boldsymbol{u}_{h,0}\|_{0}^2$.
For the second term, applying 
\cref{lem:adjoint_div_grad}
and the condition $\hat{\boldsymbol{u}} \in {\cal K}$ yields
$
I_2 = -(\boldsymbol{\sigma}_{h,0}, \stfD \hat{\boldsymbol{u}}) = 0.
$
Also, the third term vanishes due to orthogonality, and the fourth term simplifies to $I_4 = \|\hat{\boldsymbol{u}}\|_{0}^2$. Hence, we conclude that:
\begin{equation}
    \label{eq:integral_value_3D}
    (\boldsymbol{u}_h,\nabla\cdot\boldsymbol{\sigma}_h) = \|\boldsymbol{u}_{h,0}\|_0^2 + \|\hat{\boldsymbol{u}}\|_0^2.
\end{equation}

Next, we estimate $\|{\boldsymbol{u}_h}\|_0$ and $\|{\boldsymbol{\sigma}}_h\|_1$.
    For ${\boldsymbol{u}}_{h}={\boldsymbol{u}}_{h,0}+\hat{{\boldsymbol{u}}}$, we have:
    \begin{equation}
        \|{\boldsymbol{u}}_{h}\|_0^2=\|{\boldsymbol{u}}_{h,0}+\hat{{\boldsymbol{u}}}\|_{0}^2
        =\|{\boldsymbol{u}}_{h,0}\|_0^2+\|\hat{{\boldsymbol{u}}}\|_0^2.
    \end{equation}
    And for ${\boldsymbol{\sigma}}_{h}={\boldsymbol{\sigma}}_{h,0}+\hat{{\boldsymbol{\sigma}}}$, the $H^1$ stability of ${\cal I}_h$ and Cauchy inequality give:
    \begin{equation}
        \label{eq:approx_norm_psi_3D}
        \|{\boldsymbol{\sigma}}_{h}\|_{1}
        \le\|{\boldsymbol{\sigma}}_{h,0}\|_1+\|\hat{{\boldsymbol{\sigma}}}\|_1
        \lesssim\|{\boldsymbol{\sigma}}_0\|_1+\|\hat{\boldsymbol{u}}\|_0
        \lesssim\|{\boldsymbol{u}}_{h,0}\|_{0}+ \|\hat{{\boldsymbol{u}}}\|_{0}
        \lesssim \|\boldsymbol{u}_{h,0}\|_0.
    \end{equation}
    Therefore, for $\boldsymbol{u}_{h}\in {\rm U}_{h}$, there exists $\boldsymbol{\sigma}_h\in\Sigma_h$ such that
    \begin{equation}
        \frac{|(\boldsymbol{u}_h,\nabla\cdot\boldsymbol{\sigma}_h)|}{\|{\boldsymbol{\sigma}}_{h}\|_1\|{\boldsymbol{u}}_h\|_0}\gtrsim
        \frac{\|{\boldsymbol{u}}_h\|_0^2}{\|{\boldsymbol{u}}_h\|_0\|{\boldsymbol{u}}_h\|_0}=1,
    \end{equation}
    which yields the discrete inf-sup condition \cref{eq:inf_sup_for_Sigma_and_U}.
\end{proof}
 
\subsection{Coercivity on the discrete kernel}
To verify the coercivity on the discrete kernel in \cref{eq:SufficientCondition_coercive}, we first state the constraint characterizing $\ker B_h$,
\begin{equation}
\label{eq:kerBh}
\begin{aligned}
{
(\boldsymbol{v}_h,\nabla \cdot\boldsymbol{\sigma}_h)
+
(\nabla p_h,\boldsymbol{v}_h)
+
(\gamma_h, \nabla \cdot \boldsymbol{s}_h)
}
=0
\quad
\forall 
(\boldsymbol{v}_h,\gamma_h) 
\in \mathrm{U}_{h}\times \Theta_h.
\end{aligned}
\end{equation}
For any $(\boldsymbol{\sigma}_h,\boldsymbol{s}_h,p_h)\in \ker B_h$, this implies the following orthogonal condition:
\begin{align}
\label{eq:coercive_ker_orthgonal_condition}
\left(\boldsymbol{v}_h,\nabla\cdot\boldsymbol{\sigma}_h+\nabla p_h\right)=0,~\forall\, \boldsymbol{v}_h\in \mathrm{U}_{h}.    
\end{align}
Since $\ker{B_h} \not\subseteq \ker B$, coercivity does not directly follow from the continuous case. 
Specifically, for \cref{eq:coercive_ker_orthgonal_condition}, $\nabla\cdot\boldsymbol{\sigma}_h \neq -\nabla p_h$ in general. 
Consequently, our first step is to verify the bound:
\begin{align}
\label{eq:sufficient_coercive_condition}
    \|\nabla p_h\|_{0} \lesssim \|\boldsymbol{\sigma}_h\|_{1}.
\end{align}
Establishing this estimate allows us to subsequently apply Korn's inequalities to complete the proof.

To facilitate the analysis,
let $\Pi_h: L^2(\Omega;\mathbb R^d) \to \mathrm{U}_{h}$ denote the global $L^2$ projection defined by $(\boldsymbol{v}_h, \boldsymbol{u}-\Pi_h \boldsymbol{u}) = 0$ for all $\boldsymbol{v}_h\in \mathrm{U}_{h}$. The orthogonality condition \cref{eq:coercive_ker_orthgonal_condition} implies $\Pi_h\nabla p_h=-\Pi_h(\nabla\cdot\boldsymbol{\sigma}_h)$, then yielding
\begin{align}
\label{eq:coro_coercive_ker_orthgonal_condition}
\|\Pi_h \nabla p_h\|_{0}
= \|\Pi_h (\nabla\cdot \boldsymbol{\sigma}_h)\|_{0}
\leq \|\nabla\cdot \boldsymbol{\sigma}_h\|_{0}
\leq \|\boldsymbol{\sigma}_h\|_{1}.
\end{align}
Furthermore, we observe that
\begin{align}
\label{eq:equivalent_expr_L2projection}
\| \Pi_h \nabla p_h\|_{0} 
= 
\sup_{ \boldsymbol{v}_h \in \mathrm{U}_{h}} 
\frac{(\boldsymbol{v}_h,\Pi_h\nabla p_h)}
{\|\boldsymbol{v}_h\|_{0}}
= 
\sup_{\boldsymbol{v}_h\in \mathrm{U}_{h}}
\frac{(\boldsymbol{v}_h,\nabla p_h)}
{\|\boldsymbol{v}_h\|_{0}}.
\end{align}
Therefore, \cref{eq:sufficient_coercive_condition} follows from the condition below:
\begin{lemma}
\label{lem:bercovier_type_inf_sup}
   For any $p_h \in {\rm P}_h$, there holds
\begin{align}
\label{eq:bercovier_type_inf_sup}
\sup_{\boldsymbol{v}_{h} \in \mathrm{U}_{h}} \frac{(\boldsymbol{v}_{h},\nabla p_h)}{\|\boldsymbol{v}_{h}\|_{0}} \gtrsim 
\|\nabla p_h\|_0. 
\end{align}
Consequently, if the orthogonality condition \cref{eq:coercive_ker_orthgonal_condition}
holds,
then
\begin{align}
\label{eq:sufficient_coercive_condition_proved}
    \|\nabla p_h\|_{0} \lesssim \|\boldsymbol{\sigma}_h\|_{1}.
\end{align}
\end{lemma}

\begin{proof}
\cref{eq:bercovier_type_inf_sup} is the Bercovier-Pironneau type discrete inf-sup condition for the $\mathbb P_d$-$\mathbb P_{d-1}$ Taylor-Hood pair $(\mathrm{U}_{h},{\rm P}_h)$. 
The case $d=2$ was established in \cite{bercovier1979error}; 
for $d=3$, the result follows from 
the macroelement technique as given in \cite{boffi1997three},
provided every tetrahedron contains at least one vertex in the domain interior.
Then
using \cref{eq:equivalent_expr_L2projection,eq:coro_coercive_ker_orthgonal_condition,eq:bercovier_type_inf_sup}, we obtain
$
\|\boldsymbol{\sigma}_h\|_1 \geq \sup_{\boldsymbol{v}_h \in {\rm U}_h}\frac{\left(\boldsymbol{v}_h,\nabla p_h\right)}{\|\boldsymbol{v}_h\|_0}
\gtrsim \|\nabla p_h\|_0
$
and complete the proof.
\end{proof}

We are now ready to establish the coercivity on the discrete kernel from \cref{eq:sufficient_coercive_condition_proved}.
\begin{theorem}
\label{thm:coercivity_condition_finial_2D}
For any 
$\boldsymbol{\mathcal S}_h \in 
\ker B_h$, the following coercivity condition holds:
\begin{equation}
\begin{aligned}
\mathcal A(\boldsymbol{\cal S}_h,\boldsymbol{\cal S}_h)
\gtrsim 
\|\boldsymbol{\cal S}_h\|^2_{\rm T}
\coloneqq
\|\boldsymbol{s}_h\|_1^2 + 
\|\boldsymbol{\sigma}_h\|_1^2 +
\|p_h\|_1^2.
\end{aligned}
\end{equation}
\end{theorem}

\begin{proof}
It suffices to verify \cref{eq:SufficientCondition_coercive}, which means 
for any $(\boldsymbol{\sigma}_h,\boldsymbol{s}_h,p_h)\in \ker B_h$,
\begin{equation}
\label{eq:SufficientCondition_coercive_discrete}
 \begin{aligned}
      \|{\rm sym~\nabla}\boldsymbol s_h\|_0^2
      + \|\boldsymbol s_h\|_0^2 
      + \|{\rm Stf~\nabla} \boldsymbol\sigma_h\|_{0}^2 + 
      \|{\boldsymbol{\sigma}_h}\|_0^2 
      \gtrsim 
      \|\boldsymbol{s}_h\|_1^2 + 
    \|\boldsymbol{\sigma}_h\|_1^2 +
    \|p_h\|_1^2.
 \end{aligned}
\end{equation}
From Korn's inequality of the second type
and 
\cite[Lemma 3.11]{lewintan2025wellposedness},
we have 
\begin{align}
\|{\rm sym~\nabla}\boldsymbol s_h\|_0^2
+ \|\boldsymbol s_h\|_0^2
\gtrsim
\|\boldsymbol{s}_h\|_1^2,
\quad
\|{\rm Stf~\nabla} \boldsymbol\sigma_h\|_{0}^2 + \|{\boldsymbol{\sigma}_h}\|_0^2   
\gtrsim \|{\boldsymbol{\sigma}_h}\|_1^2.
\end{align}
Hence to prove \cref{eq:SufficientCondition_coercive_discrete}, it remains to show
$
\|\boldsymbol{\sigma}_h\|_1 \gtrsim \|p_h\|_1.
$
Since $p_h\in \tilde{H}^1$, the proof is completed by invoking 
Poincaré's inequality and \cref{lem:bercovier_type_inf_sup}.
\end{proof}

\subsection{The \textit{a priori} error estimate}
We are now in a position to establish an \textit{a priori} error estimate.

\begin{theorem}
\label{thm:energy_norm_estimate}
The Galerkin formulation \cref{eq:MFEM_Scheme} for the linearized R13 system admits a unique solution $(\boldsymbol{\mathcal{S}}_h,\boldsymbol{\mathcal{U}}_h) = (\boldsymbol{\sigma}_h, \boldsymbol{s}_h, p_h, \boldsymbol{u}_h, \theta_h)$, which satisfies the quasi-optimal estimate:
\begin{align}
\label{eq:MFEM_quasi_optimal_estimate}
\left\|
(\boldsymbol{\mathcal{S}}_h,\boldsymbol{{\mathcal U}}_h)
-
(\boldsymbol{\mathcal{S}},\boldsymbol{{\mathcal U}})
\right\|_{\rm T \times W} 
&\lesssim 
\inf_{
(\boldsymbol{\mathcal{R}}_h,\boldsymbol{{\mathcal V}}_h) 
\in {\rm T_h \times W_h}
}    
\big\|
(\boldsymbol{\mathcal{R}}_h,\boldsymbol{{\mathcal V}}_h)
-
(\boldsymbol{\mathcal{S}},\boldsymbol{{\mathcal U}})
\big\|_{\rm T \times W}.
\end{align}
Furthermore, assume the exact solution 
$(\boldsymbol{\mathcal{S}},\boldsymbol{{\mathcal U}})=(\boldsymbol{\sigma}, \boldsymbol{s}, p, \boldsymbol{u}, \theta)$
satisfies
\begin{equation}
    \label{eq:regularity_of_exact_solution_2D}
    \boldsymbol{\sigma} \in H^{d+1}(\Omega;{\mathbb{T}}),\quad 
    \boldsymbol{s} \in H^{d+1}(\Omega;\mathbb{R}^d),\quad 
    p\in \tilde{H}^{d+1}(\Omega),\quad
    \boldsymbol{u} \in H^{d}(\Omega;\mathbb{R}^d),\quad 
    \theta \in H^{d}(\Omega),
\end{equation}
then the following a priori error estimate holds:
\begin{equation}
\begin{aligned}
&\|\boldsymbol{\sigma}-\boldsymbol{\sigma}_h\|_1
+\|\boldsymbol{s}-\boldsymbol{s}_h\|_1
+\|{p}-{p}_h\|_1
+\|\boldsymbol{u}-\boldsymbol{u}_h\|_0
+\|{\theta}-{\theta}_h\|_0
\\
&\qquad \lesssim
h^{d-1}|{p}|_{d}+
h^{d}\big(
|\boldsymbol{\sigma}|_{d+1}
+
|\boldsymbol{s}|_{d+1}
+|\boldsymbol{u}|_{d}
+|\theta|_{d}
\big).
\end{aligned}
\end{equation}
\end{theorem}

\begin{proof}
The boundedness of the bilinear operators follows from the continuous formulation; see \cite[Appendix A]{lewintan2025wellposedness}. 
The coercivity on the discrete kernel and the discrete inf-sup conditions are established in \cref{thm:coercivity_condition_finial_2D,thm:infsup_vector,thm:inf_sup_for_Sigma_and_U}.
Consequently, the standard Babuška-Brezzi theory implies the unisolvence 
of the Galerkin formulation \cref{eq:MFEM_Scheme}
and the quasi-optimal error estimate \cref{eq:MFEM_quasi_optimal_estimate}. 
The final \textit{a priori} error bound follows directly from standard interpolation estimates.
\end{proof}

Regarding the $L^2$ norm error estimate, which is often of greater practical significance, we anticipate a gain of one order of convergence compared to the energy norm. 
Here we establish this result using a duality argument, under the assumption of \textit{elliptic regularity} for the dual problem.

\begin{corollary}
\label{cor:L2_estimate}
Let $\boldsymbol{e}_{\mathcal{S}} = \boldsymbol{\mathcal{S}} - \boldsymbol{\mathcal{S}}_h$ and $\boldsymbol{e}_{\mathcal{U}} = \boldsymbol{\mathcal{U}} - \boldsymbol{\mathcal{U}}_h$. Assume that the dual problem \cref{eq:dual_1}-\cref{eq:dual_2} (defined in the proof below) admits the elliptic regularity property
\begin{align}
\label{eq:dual_regularity}
\|\boldsymbol{T}\|_2 + \|\boldsymbol{W}\|_1 \lesssim \|\boldsymbol{e}_{\boldsymbol{\cal S}}\|_0.
\end{align}
Then, the following estimate holds:
\begin{equation}
\label{eq:L2_estimate_concrete}
\|\boldsymbol{\sigma}-\boldsymbol{\sigma}_h\|_0
+\|\boldsymbol{s}-\boldsymbol{s}_h\|_0
+\|{p}-{p}_h\|_0
\lesssim
h^{d}|{p}|_{d}+
h^{d+1}\big(
|\boldsymbol{\sigma}|_{d+1}
+
|\boldsymbol{s}|_{d+1}
+|\boldsymbol{u}|_{d}
+|\theta|_{d}
\big).
\end{equation}
\end{corollary}

\begin{proof}
Consider the dual problem: seek $(\boldsymbol{T}, \boldsymbol{W}) \in \mathrm{T} \times \mathrm{W}$ such that
\begin{alignat}{2}
\mathcal{A}(\boldsymbol{\mathcal{R}}, \boldsymbol{T}) + \mathcal{B}(\boldsymbol{\mathcal{R}}, \boldsymbol{W}) &= (\boldsymbol{e}_{\mathcal{S}}, \boldsymbol{\mathcal{R}}) \quad &\forall\, \boldsymbol{\mathcal{R}} \in \mathrm{T}, \label{eq:dual_1} \\
\mathcal{B}(\boldsymbol{T}, \boldsymbol{\mathcal{V}}) &= 0 \quad &\forall\, \boldsymbol{\mathcal{V}} \in \mathrm{W}. \label{eq:dual_2}
\end{alignat}
By testing \cref{eq:dual_1,eq:dual_2} with $\boldsymbol{\mathcal{R}}=\boldsymbol{e}_{\boldsymbol{\cal S}}$ and $\boldsymbol{\mathcal{V}} = \boldsymbol{e}_{\boldsymbol{\cal U}}$, and employing the Galerkin orthogonality, we obtain
\begin{equation}
\begin{aligned}
\|\boldsymbol{e}_{\boldsymbol{\cal S}}\|_0^2 
&= \mathcal A(\boldsymbol{e}_{\boldsymbol{\mathcal{S}}},\boldsymbol{T}-\boldsymbol{T}_h) 
- \mathcal B(\boldsymbol{T}_h,\boldsymbol{e}_{\boldsymbol{\cal U}}) 
+ \mathcal B(\boldsymbol{e}_{\boldsymbol{\cal S}},\boldsymbol{W}-\boldsymbol{W}_h)\\
&= \mathcal A(\boldsymbol{e}_{\boldsymbol{\mathcal{S}}},\boldsymbol{T}-\boldsymbol{T}_h) 
+ \mathcal B(\boldsymbol{T}-\boldsymbol{T}_h,\boldsymbol{e}_{\boldsymbol{\cal U}}) 
+ \mathcal B(\boldsymbol{e}_{\boldsymbol{\cal S}},\boldsymbol{W}-\boldsymbol{W}_h)\\
&\lesssim \left( \|\boldsymbol{e}_{\boldsymbol{\cal{S}}}\|_{\mathrm{T}} + \|\boldsymbol{e}_{\boldsymbol{\cal{U}}}\|_{\mathrm{W}} \right) \cdot
h\left( \|\boldsymbol{T}\|_2 + \|\boldsymbol{W}\|_1\right). 
\end{aligned}
\end{equation}
Applying the regularity assumption \cref{eq:dual_regularity} allows us to cancel one factor of $\|\boldsymbol{e}_{\boldsymbol{\cal S}}\|_0$, yielding
$
\|\boldsymbol{e}_{\boldsymbol{\cal S}}\|_0 \lesssim h \left( \|\boldsymbol{e}_{\mathcal{S}}\|_{\mathrm{T}} + \|\boldsymbol{e}_{\boldsymbol{\cal{U}}}\|_{\mathrm{W}} \right).
$
Combining this result with the energy norm estimates in \cref{thm:energy_norm_estimate} gives the desired bound \cref{eq:L2_estimate_concrete}.
\end{proof}

\begin{remark}
The global convergence rates in \cref{thm:energy_norm_estimate,cor:L2_estimate} are limited by the pressure error. 
Such behavior reflects a necessary trade-off between conformity and stability under a specific choice of $H^1$ pressure and $L^2$ velocity spaces in the current weak form. 
\end{remark}

\section{Numerical experiments}
\label{sec:numerical_results}

This section presents numerical results for the MFEM \cref{eq:MFEM_Scheme} with finite element spaces in \cref{eq:FESpaces}. We demonstrate the stability and convergence of the method, with $\tilde{\chi}=1$ for all test cases. The scheme is implemented using the \texttt{MATLAB} package \texttt{ifem} \cite{chen2009ifem}.

\subsection{Homogeneous flow around a cylinder}
We first test Couette-Fourier flows within an annular domain $\Omega\subset\mathbb{R}^2$ defined by
\begin{equation}
    \label{eq:ring_domain}
    \Omega:=\left\{\mathbf{x}=(x,y)^{T}:R_1\le \|\mathbf{x}\|_2\le R_2\right\},
\end{equation}
with $R_1=0.5$ and $R_2=2$. 
Simulations are performed on body-fitted triangular meshes generated via the \texttt{distmesh2d} function in \texttt{ifem}. 

\begin{example}
This case simulates Couette-Fourier flow driven by two uniformly rotating cylinders held at constant temperatures. The parameters are defined as:
\begin{equation}
\label{eq:params_case_1}
\begin{aligned}
&u_{t}^{W}\big|_{\|\mathbf{x}\| = R_1}=u_{t}^{W}\big|_{\|\mathbf{x}\| = R_2}=1,\quad \theta^{W}\big|_{\|\mathbf{x}\| = R_1}=\theta^{W}\big|_{\|\mathbf{x}\| = R_2}=1,\\
&\text{Kn}=0.05,0.1,0.2,0.4,\quad
h=0.2, 0.1, 0.05, 0.025,0.0125,\quad\text{respectively}.
\end{aligned}
\end{equation}
We define the \textit{absolute $L^2$ error} as
\begin{equation}
    e_{f}:=\|f_{\text{exact}}-f_h\|_{0}.
\end{equation}
The analytical solution is derived from \cite[Appendix A]{saravanakumar2020investigation}. \Cref{fig:Error_Estimation} plots the errors against the mesh size $h$. The results confirm the second-order convergence of the proposed scheme across various Knudsen numbers.
\begin{figure}[!htbp]
    \centering
    \includegraphics[width=0.95\linewidth]{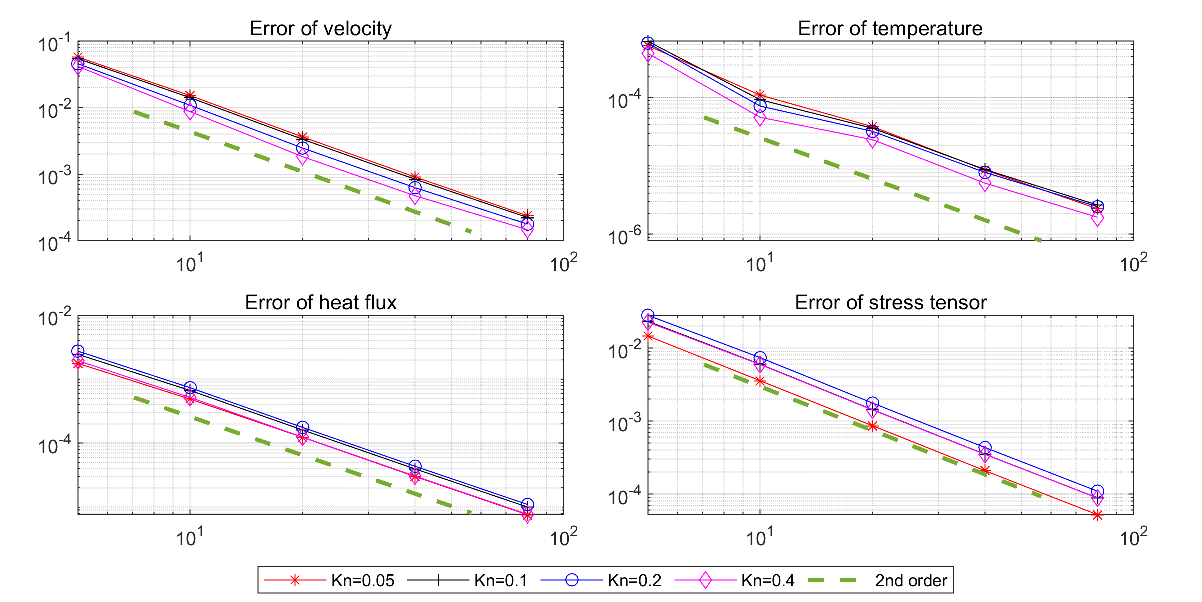}
    \caption{The absolute $L^2$ error for each unknown in the R13 system.}
    \label{fig:Error_Estimation}
\end{figure}
\end{example}

\begin{example}
Next, we investigate the Couette-Fourier flow where the outer wall temperature differs from the inner wall. The parameters are: 
\begin{equation}
\label{eq:params_case_2}
\begin{aligned}
&u_{t}^{W}\big|_{\|\mathbf{x}\| = R_1}=u_{t}^{W}\big|_{\|\mathbf{x}\| = R_2}=1,\quad \theta^{W}\big|_{\|\mathbf{x}\| = R_1}=1,\quad \theta^{W}\big|_{\|\mathbf{x}\| = R_2}=2,
\\
&\text{Kn}=0.1,\quad h=0.1.
\end{aligned}
\end{equation}

We compare the proposed scheme with the unenriched $\mathbb{P}_2$-$\mathbb{P}_2$-$\mathbb{P}_1$-$\mathbb{P}_2$-$\mathbb{P}_1$ element. 
Since the unenriched formulation yields a singular linear system, we pick the solution with the minimum norm.
The resulting $\theta$, $u_x$ and $u_y$ profiles are presented in \cref{fig:case2}. As detailed in \cref{fig:case2_fig2}, the velocity field exhibits spurious oscillations.
This confirms the necessity of stabilizing $\Sigma_h$ via enrichment with bubble functions,
or the penalty terms proposed in \cite{theisen2021fenics}.
\begin{figure}[!htbp]
\centering
\subfloat[Numerical solutions for $\theta, u_x, u_y$ by $\mathbb{P}_2^b$-$\mathbb{P}_2$-$\mathbb{P}_1$-$\mathbb{P}_2$-$\mathbb{P}_1$ element.]
{\includegraphics[width=0.95\textwidth]{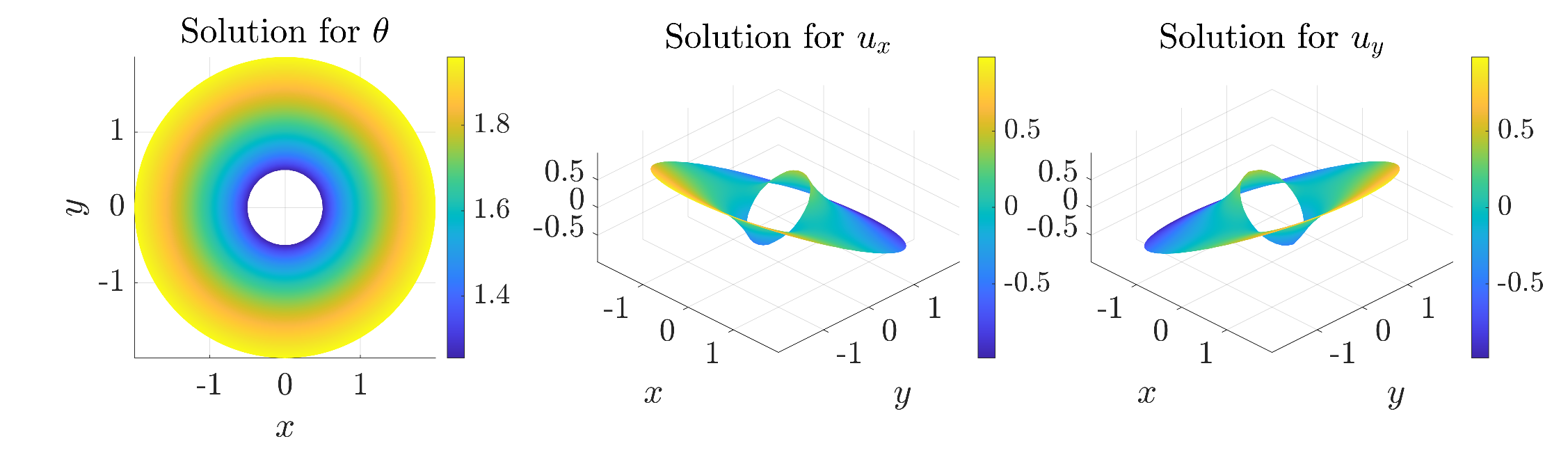} 
\label{fig:case2_fig1}
}\\       
\subfloat[Numerical solutions for $\theta, u_x, u_y$ by  $\mathbb{P}_2$-$\mathbb{P}_2$-$\mathbb{P}_1$-$\mathbb{P}_2$-$\mathbb{P}_1$ element.]
{\includegraphics[width=0.95\textwidth]{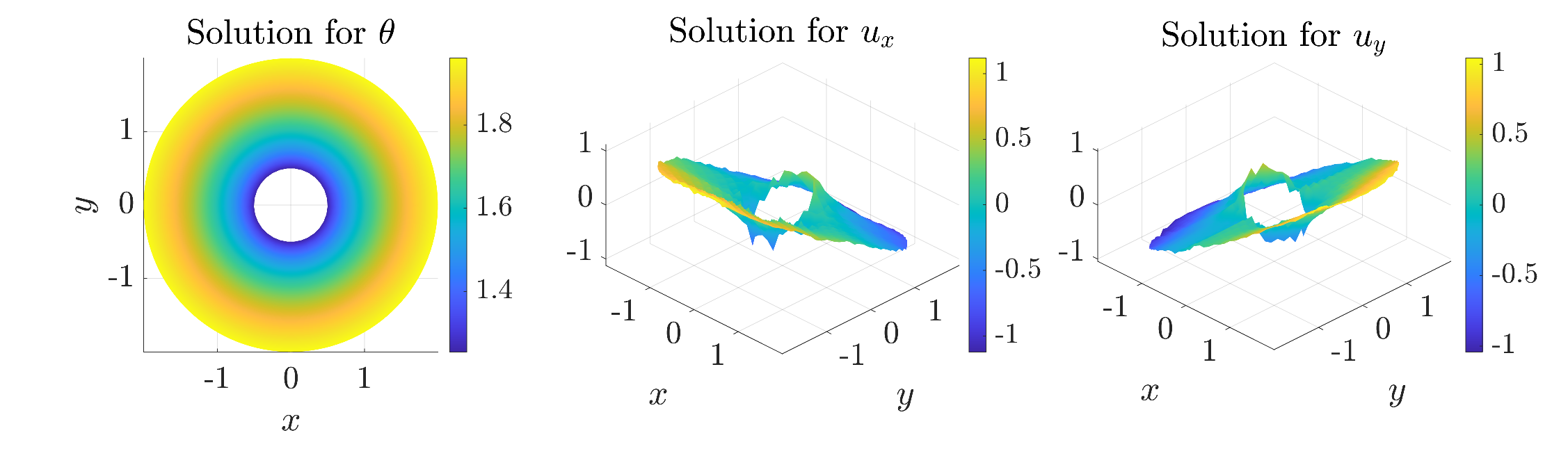} 
\label{fig:case2_fig2}
}
\caption{Temperature and velocity distributions for Couette-Fourier flow around a cylinder.}
\label{fig:case2}
\end{figure}    

\end{example}

\subsection{Channel flow within a square cavity}
We consider the channel flow within a square cavity. 
The problem domain is the unit square $\Omega:=(0,1)^2$ using uniform triangulations.
The boundary $\partial\Omega$ is partitioned into the bottom wall $\Gamma_1$ and the remainder $\Gamma_2$, with
$\Gamma_1:=\{(x,y):y=0,\, x\in(0,1)\}$, $\Gamma_2:=\partial\Omega\setminus\Gamma_1$.
Note that singularities will arise in the solution at the lower corners.

\begin{example}
We simulate Fourier flow with fixed walls ($u_{t}^{W}=0$). 
The bottom boundary $\Gamma_1$ is heated, while $\Gamma_2$ is held at a lower temperature.  
We vary the Knudsen number $\text{Kn}$ to analyze heat conduction regimes using the following parameters:
\begin{equation}
    \label{eq:params_case_3}
    \begin{aligned}
&u_{t}^{W}|_{\Gamma_1}=u_{t}^{W}|_{\Gamma_2}=0,\quad \theta^{W}|_{\Gamma_1}=1,\quad\theta^{W}|_{\Gamma_2}=0,\\
    &\text{Kn}=0.01, 0.05, 0.2,\text{ respectively,} \quad h=0.02.
    \end{aligned}
\end{equation}

\begin{figure}[!htbp]
\centering
\includegraphics[width=0.95\linewidth]{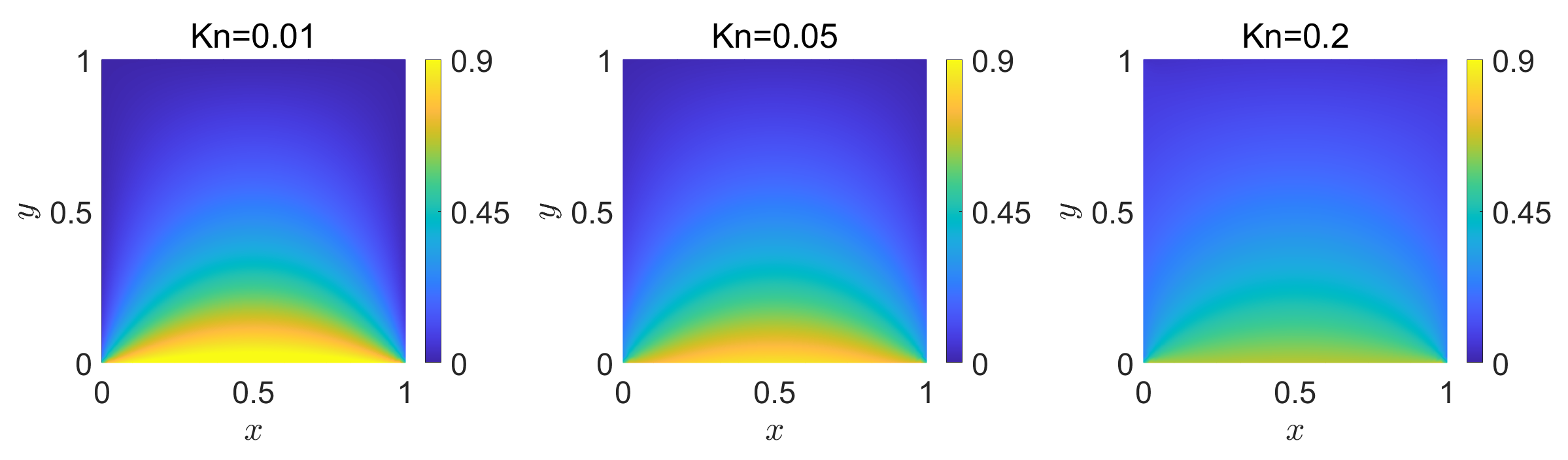}
\caption{Temperature distributions for Fourier flow in a square cavity.}
\label{fig:case3}
\end{figure}

\Cref{fig:case3} displays the temperature distributions for various values of $\text{Kn}$. As expected physically, heat transfer is more efficient in the continuum regime (lower $\text{Kn}$), resulting in deeper thermal propagation from $\Gamma_1$. The proposed scheme accurately captures this behavior, demonstrating its robustness in resolving the physics of the small $\text{Kn}$ limit.
\end{example}

\subsection{Thermally-induced edge flow}
\begin{example}
Finally, we investigate the thermally-induced edge flow, a benchmark problem from \cite[Section 6.3]{theisen2021fenics}. 
The computational domain $\Omega$ is a square with an inner square obstacle defined by $\Omega:=(0,8)^2\setminus [1,3]^2$.
Both the inner and outer boundaries are stationary no-slip walls. A temperature difference is imposed such that the outer boundary is maintained at a lower temperature than the inner boundary, and we set $\text{Kn}=0.001$. 
This choice falls within the near-continuum regime and is used to validate the asymptotic preservation properties of kinetic schemes \cite{su2020implicit,su2020fast}.
The parameters are summarized as:
\begin{equation}
\label{eq:params_case_5}
\begin{aligned}
&u_{t}^{W}|_{\text{inner}}=u_{t}^{W}|_{\text{outer}}=0,\quad \theta^{W}|_{\text{inner}}=0,\quad\theta^{W}|_{\text{outer}}=1,\quad\text{Kn}=0.001.
\end{aligned}
\end{equation}

First, we present the shear stress and thermal fields computed with a mesh size of $h=0.01$ in \cref{fig:case4}. The left panel displays a stable shear stress field $\sigma_{12}$, and the right panel depicts temperature $\theta$ contours and heat flux $\boldsymbol{s}$ streamlines consistent with Fourier's law.

\begin{figure}[!htbp]
\centering
\includegraphics[width=0.95\linewidth]{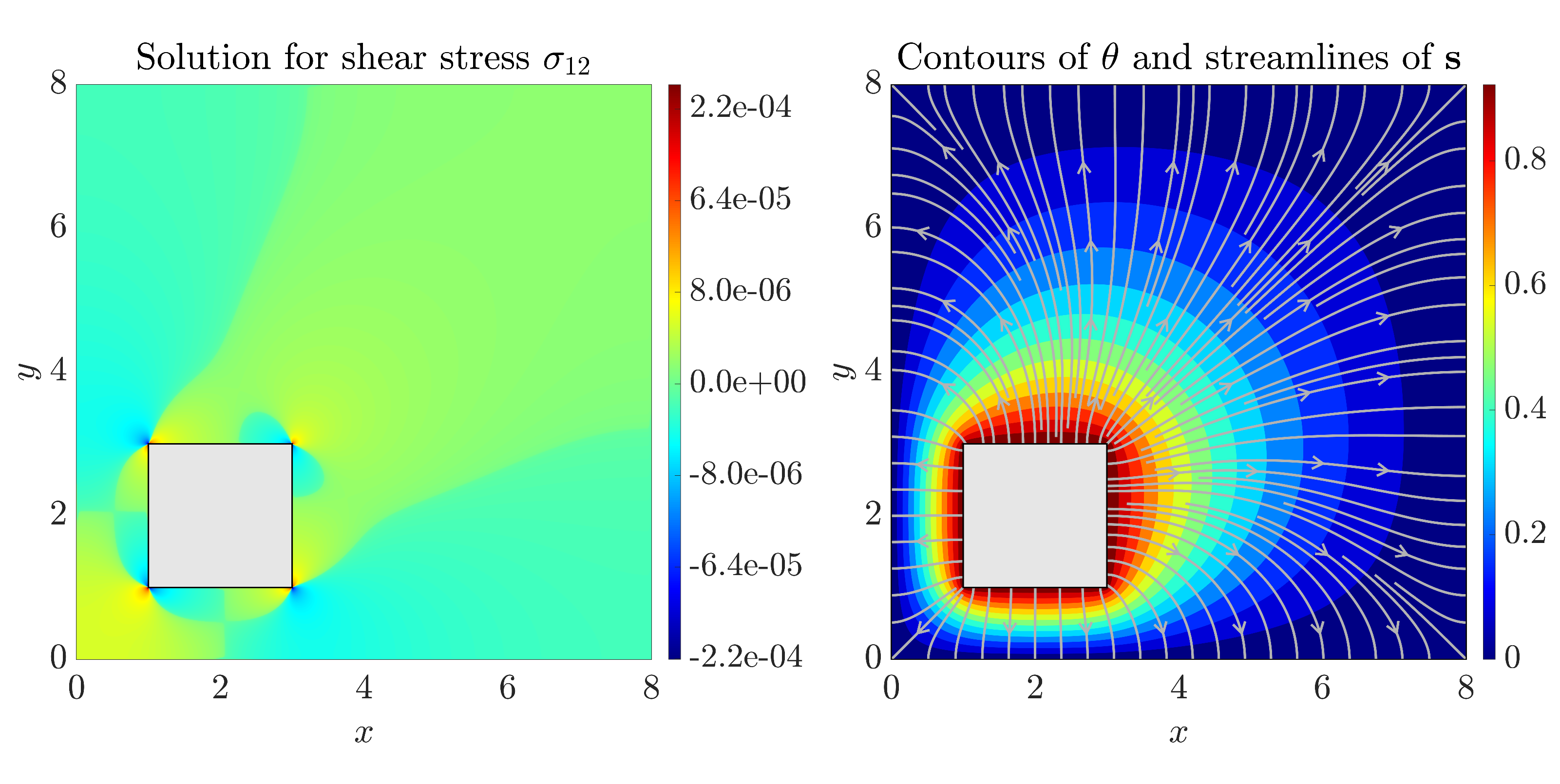}
\caption{Shear stress and heat distributions for the thermally-induced edge flow.}
\label{fig:case4}
\end{figure}

Next, we assess numerical stability by comparing our proposed scheme against the standard Taylor-Hood element ($\mathbb{P}_2$-$\mathbb{P}_1$-$\mathbb{P}_1$-$\mathbb{P}_2$-$\mathbb{P}_1$) on uniform meshes with $h=0.02$ and $h=0.01$. 
As illustrated in \cref{fig:case5}, the Taylor-Hood element exhibits significant instability, 
characterized by spurious checkerboard modes in the velocity magnitude. 
In contrast, the proposed scheme remains robust and stable under identical conditions, 
accurately capturing the rarefaction effects inherent in the edge flow.

\begin{figure}[!htbp]
\centering
\subfloat[Velocity magnitudes computed by $\mathbb{P}_2^b$-$\mathbb{P}_2$-$\mathbb{P}_1$-$\mathbb{P}_2$-$\mathbb{P}_1$ element.]
{\includegraphics[width=0.95\textwidth]{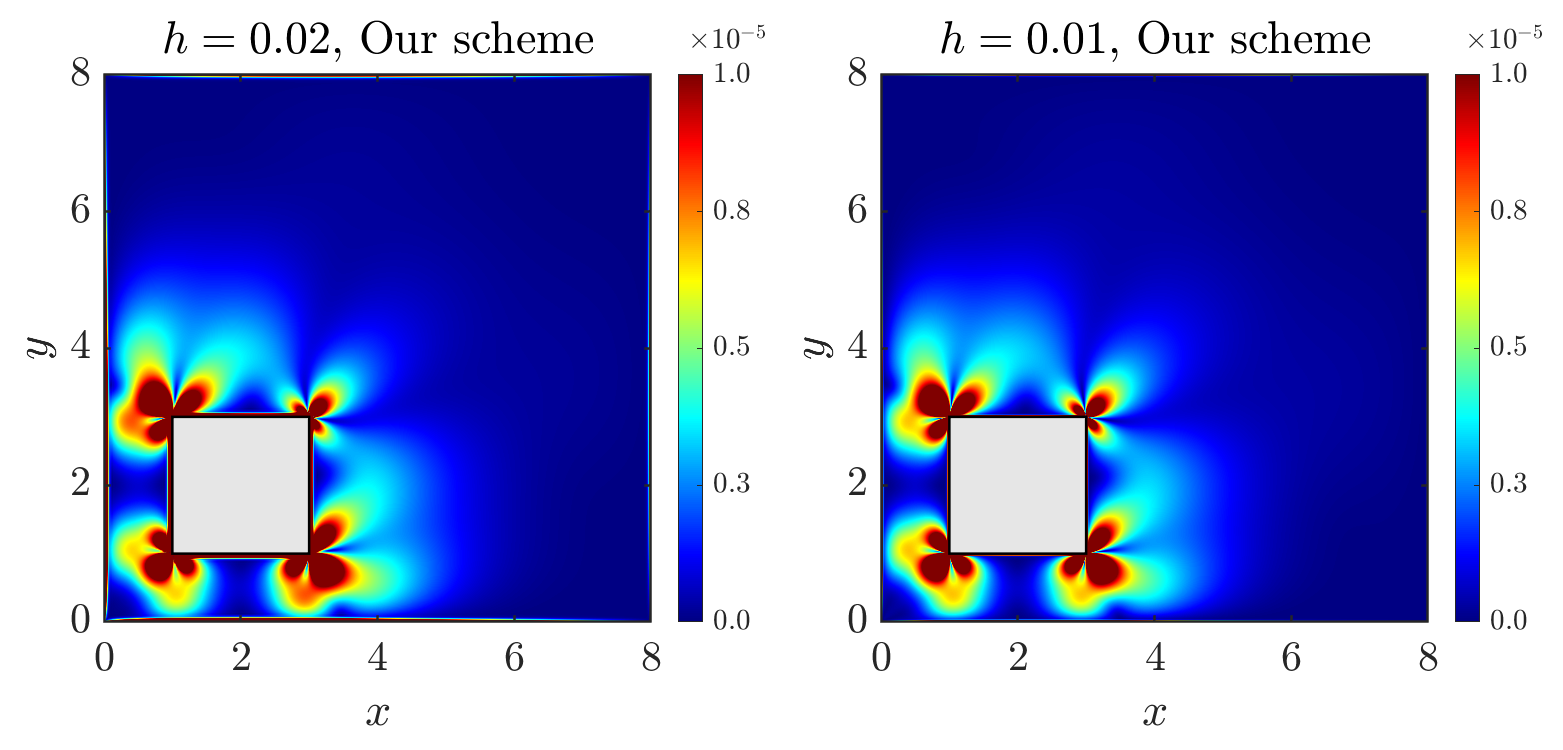}
\label{fig:case5_fig1}
}\\       
\subfloat[Velocity magnitudes computed by   $\mathbb{P}_2$-$\mathbb{P}_1$-$\mathbb{P}_1$-$\mathbb{P}_2$-$\mathbb{P}_1$ element.]
{\includegraphics[width=0.95\textwidth]{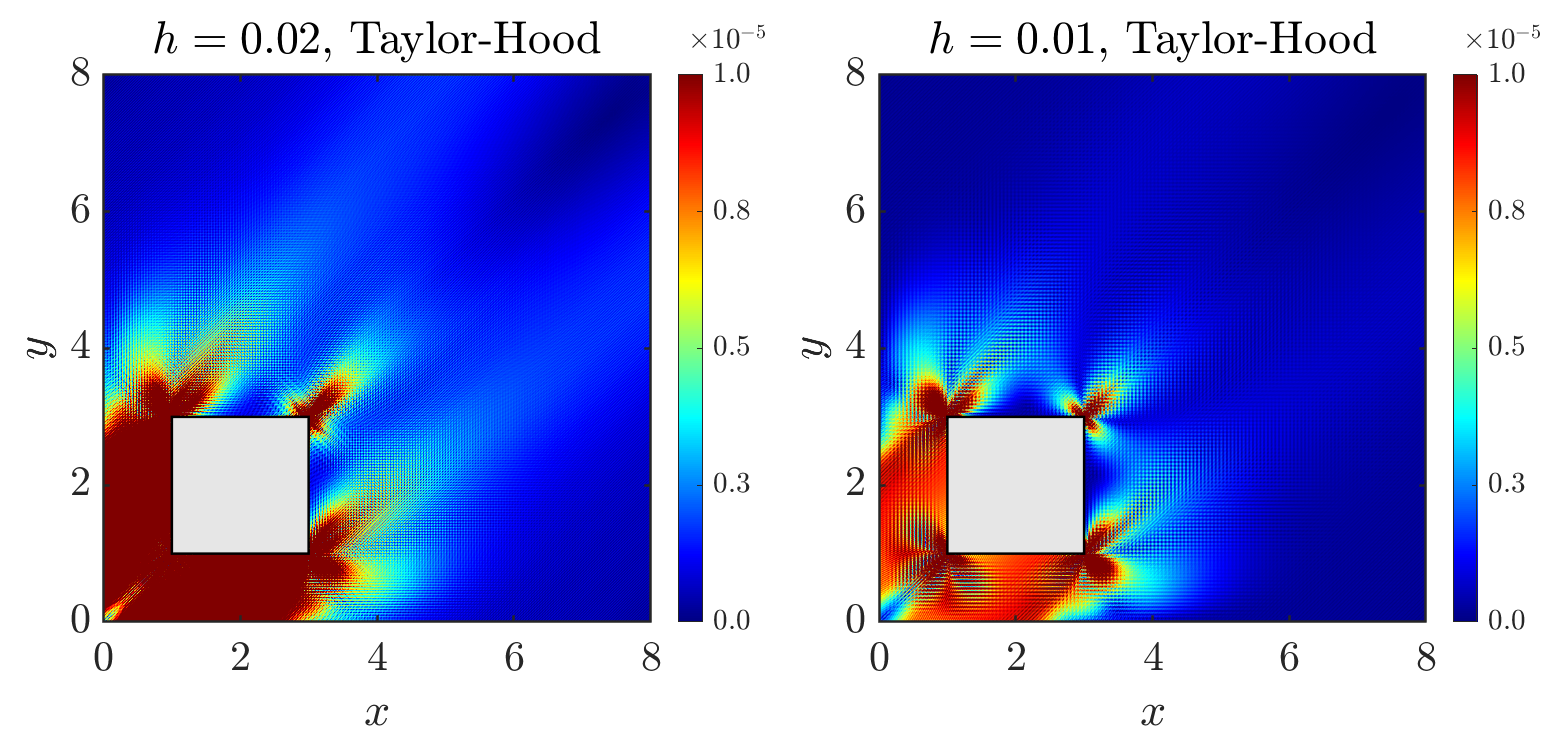} 
\label{fig:case5_fig2}
}
\caption{Velocity magnitude distributions for the thermally-induced edge flow.} 
\label{fig:case5}
\end{figure}  

Finally, we take a closer look at the solution on the fine mesh with $h=0.01$. 
\Cref{fig:case6-3} displays the velocity streamlines and profiles of vertical velocity $u_y$ at $y=0.5$ and $y=4.5$. 
While the standard Taylor-Hood element produces pronounced oscillations, our scheme successfully resolves the flow structure, demonstrating its effectiveness and accuracy in this challenging regime.

\begin{figure}[!htbp]
\centering
\subfloat[Stable velocity field computed by  $\mathbb{P}_2^b$-$\mathbb{P}_2$-$\mathbb{P}_1$-$\mathbb{P}_2$-$\mathbb{P}_1$ element.]
{
    \label{fig:case6_fig1}
    \includegraphics[width=0.95\textwidth]{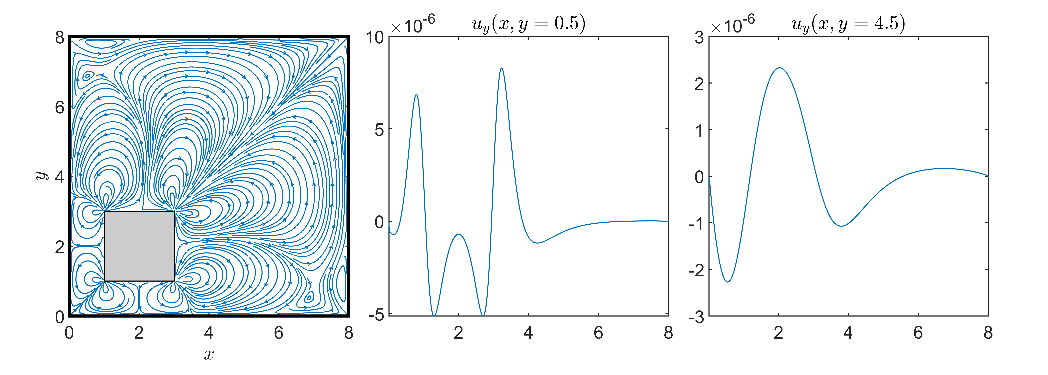}
} \\ 
\subfloat[Unstable velocity field computed by $\mathbb{P}_2$-$\mathbb{P}_1$-$\mathbb{P}_1$-$\mathbb{P}_2$-$\mathbb{P}_1$ element.]
{
    \label{fig:case6_fig2}
    \includegraphics[width=0.95\textwidth]{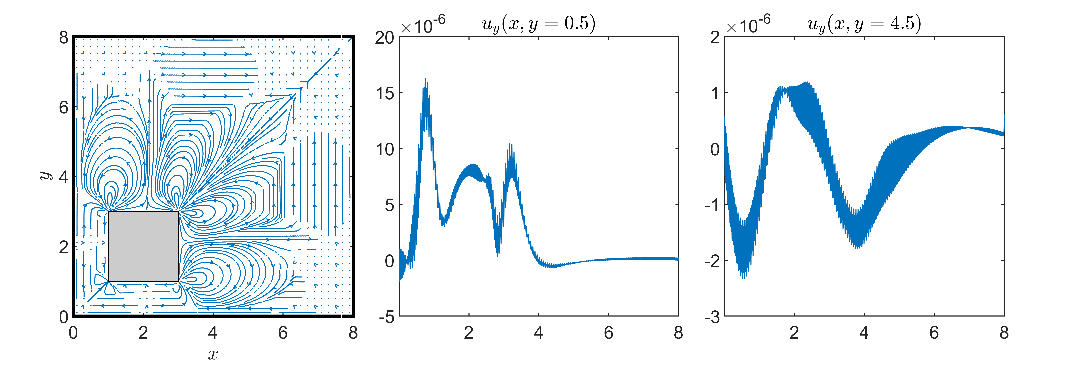}
}
\caption{Velocity streamlines and $u_y$ profiles for the thermally-induced edge flow.}
\label{fig:case6-3}
\end{figure}

\end{example}

\section{Conclusions}
\label{sec:conclusions}
In this paper, we have proposed and analyzed a novel conforming MFEM for the linear R13 equations. By enriching the tensor element space with specific bubble functions, we successfully establish a stable, convergent, and penalty-free scheme with rigorous theoretical analysis.

The proposed framework offers significant advantages for further research. First, our scheme can be extended to R13 systems with inflow or outflow boundary conditions. Second, for time-dependent R13 systems, as discussed in \cite{lin2025time}, designing stable schemes using CIP methods can be impractical due to parameter tuning; our penalty-free approach provides a more natural and robust foundation for such extensions. 
Third, the inherent stability of our scheme makes it highly suitable for simulating complex rarefied gas phenomena.

Given that the enriched space increases the number of degrees of freedom, future work will focus on developing efficient preconditioners for fast solvers and reduced-basis methods to mitigate computational costs. Additionally, extending this rigorous MFEM framework to time-dependent cases or to other R13 systems for general gas molecules, as described in \cite{Cai2024Onsager}, is a promising direction for future investigation.

\appendix
\section{Korn-type inequalities}
\label{sec:appendixA}

In this section, we review $\mathbb{C}$-ellipticity for first-order linear homogeneous differential operators and prove \cref{thm:Korn_inequality_stfD_negative_norm}.

\begin{definition}
\label{defn:ComplexEllip_3in1}
Let $V=\mathbb R^d$ and $W=\mathbb R^{d'}$.
A first-order linear homogeneous operator with constant coefficients
$
\mathbb{A} := \sum_{j=1}^{d} \mathbb{A}_{j} \partial_{j},
$
where $\mathbb{A}_{j} \in L(V, W)$,
is called $\mathbb{C}$-elliptic, if the corresponding symbol map 
$\mathbb A[\xi]:\mathbb C^d \to \mathbb C^{d'}$,
$
\mathbb{A}[\xi](v) := \sum_{j=1}^{d} \xi_{j} \mathbb{A}_{j}v
$
$(\xi \in  \mathbb C^d)$
has no common non-trivial complex zero.
That is,
\begin{equation}\label{eq:Complex_Elliptic}
\forall \xi \in \mathbb{C}^{d}, \quad
\mathbb A[\xi](v)={0}
\quad \Longrightarrow \quad v = {0}.
\end{equation}
\end{definition}

Then we have the following ellipticity condition for the operator $\mathrm{stf}\,\nabla$.

\begin{theorem}
    \label{thm:stfD_3D_C_elliptic_1_tensor}
    The operator $\mathrm{stf}\,\nabla$ acting on vector fields is $\mathbb{C}$-elliptic if and only if the dimension $d\ge 3$.
\end{theorem}

\begin{proof}
To establish the $\mathbb{C}$-ellipticity of $P:=\mathrm{stf}\,\nabla$, it suffices to show that its symbol is injective for any $\xi\in\mathbb{C}^d \setminus \{0\}$. As noted in \cite[Remark 3.4]{lewintan2025wellposedness}, the symbol is given by $P[\xi](v)=\mathrm{stf}(v\otimes\xi)$; therefore, injectivity is equivalent to the condition:
\begin{equation}
\label{eq:Cellipforvec}
\mathrm{stf}(v\otimes \xi)=0 \quad \Longrightarrow \quad v=0.
\end{equation}
Let $v=(v_{i})$ and $\xi=(\xi_{j})$. The condition \eqref{eq:Cellipforvec} implies:
\begin{equation}
\label{eq:Cellip_expanded}
\begin{aligned}
v_{i}\xi_{i} = v_{j}\xi_{j},~ 
v_{i}\xi_{j} = -v_{j}\xi_{i},\quad\forall i\neq j.
\end{aligned}
\end{equation}

If $d=2$, the vectors $v=(i,1)^T$ and $\xi = (i,-1)^{T}$ serve as a counterexample for \eqref{eq:Cellip_expanded}.
If $d\ge 3$, consider two cases. First, suppose $\xi$ contains zero components. Up to permutation, let $\xi_1=\xi_{2}=\ldots=\xi_{j}=0$ and $\xi_{l}\neq 0$ for all $l>j$. Since $v_{i}\xi_{i} \equiv v_{l}\xi_{l}$ implies $v_{l}\xi_{l}=0$ (as we can pick an index from the zero components), and $\xi_l \neq 0$, it follows that $v_{l}=0$ for all $l>j$. For the indices $k \le j$, we have $v_{k}\xi_{j+1} = -v_{j+1}\xi_{k}$. Since $v_{j+1}=0$ and $\xi_{j+1}\neq 0$, this implies $v_{k}=0$ and thus $v=0$.
Second, assume all components of $\xi$ are non-zero. Then the latter equation in \eqref{eq:Cellip_expanded} combined with the first implies that for any $i\neq j$, $v_i^2 \xi_j = -v_j \xi_i v_i = -v_i \xi_i v_j = -v_j \xi_j v_j$, which simplifies to $v_{i}^2+v_{j}^2=0$ and leads to $v=0$.
\end{proof}

The following theorem shows the relationship between $\mathbb{C}$-ellipticity and Korn-type inequalities.
\begin{theorem}
\label{thm:Korn2nd}
Let $\Omega\subset\mathbb{R}^3$ be a bounded Lipschitz domain, $\mathbb{A}$ is a first-order linear homogeneous operator with constant coefficients, then the following inequality holds if and only if $\mathbb{A}$ is a $\mathbb{C}$-elliptic operator:
\begin{equation}
\|\mathbf{u}\|_s\lesssim \|\mathbf{u}\|_{t}+\|\mathbb{A}\mathbf{u}\|_{s-1},\quad{\rm for\,} t<s.
\end{equation}
\end{theorem}
\begin{proof}
    See \cite[Corollary 7.3]{smith1970formulas}.
\end{proof}

Finally, we present the proof for the generalized Korn's inequality of a negative norm.
\begin{theorem}
\label{thm:Korn_inequality_stfD_negative_norm}
For a bounded Lipschitz domain $\Omega\subset \mathbb R^3$
and
$\boldsymbol{v}\in L^2(\Omega;\mathbb R^3)$, the generalized Korn's inequality of negative norm holds:
\begin{align}
\|\boldsymbol{v}\|_0 \lesssim \|{\rm stf~\nabla}\boldsymbol{v}\|_{-1} + \|\boldsymbol{v}\|_{-1}.    
\end{align}
\end{theorem}
\begin{proof}
    The operator $P \coloneqq \mathrm{stf}\,\nabla$ is a first-order, linear homogeneous differential operator. As established in 
    \cref{thm:stfD_3D_C_elliptic_1_tensor}, $P$ is $\mathbb{C}$-elliptic for $d \ge 3$.
    Consequently, we can directly apply Korn-type inequality in \cref{thm:Korn2nd}. This corollary provides the following a priori estimate for $u\in H^s(\Omega)$:
    \begin{align}
        \|u\|_{s} \lesssim \|P u\|_{{s-1}} + \|u\|_{t},    
    \end{align}
    where $t < s$ is an arbitrary lower-order index. Setting $s=0$ and $t=-1$ yields the desired result. 
\end{proof}

\section*{Acknowledgments}
The authors would like to thank Ziwen Gu from Shanghai University of Finance and Economics for his valuable suggestions regarding conforming tensor-valued elements.

\bibliographystyle{siamplain}
\bibliography{references}
\end{document}


\maketitle

\section{Interpolation operator}
\label{sec:AppendixB}
In this section, we prove \cref{Lem:interp_op} for $d=2$. First, we derive the dual basis for $\Sigma_{h,0}$ with corresponding DoFs \eqref{eq:Sigmah_dofs}. For each element $K\subset\Omega$, we mark DoFs on element $K$ in \cref{fig:DoFs} and derive its dual basis.

\begin{figure}[!htbp]
\centering
\subfloat[Degrees of freedom on an element $K$.]
{
\includegraphics[width=0.3\textwidth]{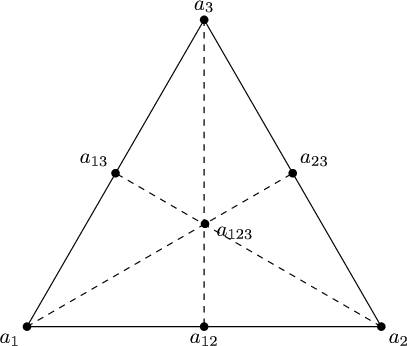}
\label{fig:DoFs}
} 
\subfloat[The patch of edge $f$.]
{
\includegraphics[width=0.3\textwidth]{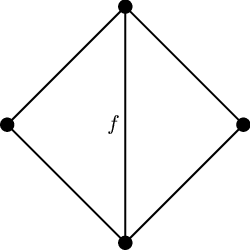}
\label{fig:patchf}
} 
\subfloat[The patch of node $z$.]
{
\includegraphics[width=0.3\textwidth]{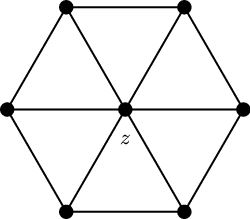}
\label{fig:patchz}
} 
\caption{Figures related to DoFs and dual basis.}
\label{fig:three_plots}
\end{figure}

Since $\tau_h\in\Sigma_{h}$ should be a symmetric $2\times 2$ tensor, the dual basis of $\Sigma_h$ shows the following form:
\begin{equation}
    \left\{\begin{bmatrix}
        u&0\\
        0&0\\
    \end{bmatrix},\quad\begin{bmatrix}
        0&u\\
        u&0\\
    \end{bmatrix},
    \quad
    \begin{bmatrix}
        0&0\\
        0&u\\
    \end{bmatrix}
    \right\},
\end{equation}
and the dual basis functions on each DoF:
\begin{itemize}
    \item On nodal points $a_i$, the basis functions:
    \begin{equation}
        \label{eq:basis_on_nodals}
        u_i=\lambda_i(3\lambda_i-2)+\lambda_1\lambda_2\lambda_3(24-42\lambda_i).
    \end{equation}
    \item On edge $a_{ij}$, the basis functions:
    \begin{equation}
        \label{eq:basis_on_edges}
        u_{ij}=\frac{6}{|a_{ij}|}(\lambda_i\lambda_j+\lambda_1\lambda_2\lambda_3(21\lambda_{k}-12)),
    \end{equation}
    where $k\neq i,j$.
    \item On the element $K$ (marked as $a_{123}$), the basis function (e.g., the first one):
    \begin{equation}
        \label{eq:basis_on_element}
        u_{123,1}=\frac{1}{|K|}\lambda_1\lambda_2\lambda_3(900\lambda_1-360\lambda_2-360\lambda_3),
    \end{equation}
    where $|K|$ denotes the area of element $K$.
\end{itemize}

From \eqref{eq:basis_on_nodals}, \eqref{eq:basis_on_edges} and \eqref{eq:basis_on_element}, we can construct the global ``hat-basis" for $\Sigma_h$. The \textbf{patch} of a node $z$ or an edge $f$, marked as $\mathcal{T}_{z}$ or $\mathcal{T}_{f}$, consists of all the elements $K$ which share common node $z$ (or edge $f$), see \cref{fig:patchf}, \cref{fig:patchz}. Then the entry of the global basis functions of node $z$, edge $f$ and element $K$ can write as
\begin{equation}
    \label{eq:global_basis}
    \phi_{z}:=\sum_{K\in\mathcal T_{z}}\chi_Ku_{K,z},
    \quad \tau_{f}:=\sum_{K\in\mathcal{T}_f}\chi_Ku_{K,f},
    \quad
    \varphi_{K,i}:=\chi_Ku_{K,i}\,(1\leq i \leq 3),
\end{equation}
respectively. Now, we show the proof for \cref{Lem:interp_op}.
\begin{proof}[Proof for \cref{Lem:interp_op}]
First, we show the operator $\mathcal{I}_h$ is a projection onto $\Sigma_h$. It suffices to show that $\mathcal{I}_h$ is the identity on $\Sigma_h$. By the definition, if $\tau\in\Sigma_h$, we have $Q_{K}\tau=\tau$ always holds on $K$. Then \eqref{eq:interp_operator} yields:
    \begin{equation}
        \label{eq:Dofs_of_Ih_tau}
        \begin{aligned}
        I_{h}\tau(\delta)&=\tau(\delta),\\
        \quad (\mathcal{I}_h\tau,\boldsymbol{q})_f&=(\tau,\boldsymbol{q})_f\quad\forall\boldsymbol{q}\in \mathbb{P}_0,\\
        (\mathcal{I}_h\tau,\boldsymbol{q})_{K}&=(\tau,\boldsymbol{q})_{K}\quad\forall\boldsymbol{q}\in\epsilon(\mathbb{P}_1(K;\mathbb{R}^2)),
        \end{aligned}
    \end{equation}
    since the DoFs in \eqref{eq:Dofs_of_Ih_tau} are identical with \eqref{eq:Sigmah_dofs}, it means that $\mathcal{I}_h\tau=\tau$ on $\Sigma_h$ and thus $\mathcal{I}_h^2=\mathcal{I}_h$, i.e., $\mathcal{I}_h$ is a projection. 

    Next, we will show the second result: $L^2$-estimation. We consider the estimation on the element $K\subset\Omega$. $\tilde{K}$ denotes the union of all mesh elements that intersect $K$, which is called the \textbf{first-order patch} of $K$. 

    \begin{figure}[htbp]
        \centering
        \includegraphics[width=0.3\linewidth]{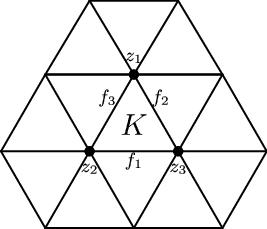}
        \caption{The first-order patch $\tilde{K}$ of element $K$.}
        \label{fig:1st_patch}
    \end{figure}
    
    By Poincaré's inequality, we can choose a constant $C_{\tilde{K}}:=\tfrac{1}{|\tilde{K}|}\int_{\tilde{K}}\tau$ such that 
    \begin{equation}
    \label{eq:Poincare}
        \|\tau-C_{\tilde{K}}\|_{0,\tilde{K}}\lesssim h|\tau|_{1,\tilde{K}}.
    \end{equation}
    Mark $w:=\tau-C_{\tilde{K}}$, since $I_{h}C_{\tilde{K}}=C_{\tilde{K}}$, we have:
    \begin{equation}
        \label{eq:L2err_of_K}
        \|\tau-I_{h}\tau\|_{0,K}\le\|\tau-C_{\tilde{K}}\|_{0,K}+\|\mathcal{I}_h(\tau-C_{\tilde{K}})\|_{0,K}
        =\|w\|_{0,K}+\|\mathcal{I}_hw\|_{0,K}.
    \end{equation}
    For $\|w\|_{0,K}$, $K\subset\tilde{K}$ implies that
    \begin{equation}
        \|w\|_{0,K}\le\|w\|_{0,\tilde{K}}\lesssim h|\tau|_{1,\tilde{K}}=h|w|_{1,\tilde{K}},
    \end{equation}
    where we use the Poincaré inequality with a sharp constant.
    
    For $\|\mathcal{I}_hw\|_0$, without loss of generality, we consider the expression for $\|\mathcal{I}_hw_{11}\|$. From \eqref{eq:global_basis}, write \begin{equation}
    \label{eq:Ih_with_basis}
    (\mathcal{I}_hw_{11})|_K=\sum_{i=1}^{3} a_i\phi_{z_i}+\sum_{i=1}^{3}b_i\tau_{f_i}+c\varphi_{K}.
    \end{equation}
    From \eqref{eq:interp_operator}, we have:
    \begin{equation}
    \label{eq:Coefficients_expansion}
        \begin{aligned}
            a_i&=\frac{1}{\# \mathcal{T}_{z_i}}\sum_{K\in\mathcal{T}_{z_i}}(Q_Kw_{11})(z_i),\\
            b_i&=\frac{1}{\#\mathcal{T}_{f_i}}\sum_{K\in\mathcal{T}_{f_i}}(Q_Kw_{11},1)_{0(f_i)},\\
            c&=(w_{11},1)_{K}.
        \end{aligned}
    \end{equation}
    Therefore:
    \begin{equation}
        \|\mathcal{I}_hw_{11}\|_{0,K}\le\sum_{i=1}^3|a_i|\|\phi_{z_i}\|_{0,K}+\sum_{i=1}^3|b_i|\|\tau_{f_i}\|_{0,K}+|c|\|\varphi_{K}\|_{0,K}:=J_1+J_2+J_3.
    \end{equation}
    By the norm equivalence on a finite-dimensional space and scaling argument, we have:
    \begin{equation}
        \label{eq:estimate_a_i}
        \begin{aligned}
            |a_i|&=\frac{1}{\#\mathcal{T}_{z_i}}\left|\sum_{K\in\mathcal{T}_{z_i}}(Q_{K}w_{11})(z_i)\right|\le\frac{1}{\#\mathcal{T}_{z_i}}\sum_{K\in\mathcal{T}_{z_i}}\|Q_{K}w_{11}\|_{L^{\infty}(K)}\\
            &\lesssim\frac{1}{\#\mathcal{T}_{z_i}}\sum_{K\in\mathcal{T}_{z_i}}h^{-1}\|Q_{K}w_{11}\|_{0,K}\lesssim h^{-1}\|Qw_{11}\|_{0,\mathcal{T}(z_i)}\\
            &\le h^{-1}\|Qw_{11}\|_{0,\tilde{K}}\le h^{-1}\|w_{11}\|_{0,\tilde{K}}\lesssim|\tau|_{1,\tilde{K}},
        \end{aligned}
    \end{equation}
    and 
    \begin{equation}
        \label{eq:estimate_phi_z_i}
        \|\phi_{z_i}\|_{0,K}=\|\lambda_i(3\lambda_i-2)+\lambda_1\lambda_2\lambda_3(24-42\lambda_i)\|_{0,K}\lesssim h,
    \end{equation}
    \eqref{eq:estimate_a_i} and \eqref{eq:estimate_phi_z_i} mean $J_1\lesssim h|\tau|_{1,\tilde{K}}$. 

    Second, by the discrete trace theorem and inverse inequality, we have:
    \begin{equation}
        \label{eq:estimate_b_i}
        \begin{aligned}
            |b_i|&=\frac{1}{\#\mathcal{T}_{f_i}}\left|\sum_{K\in\mathcal{T}_{f_i}}(Q_{K}w_{11},1)_{0(f_i)}\right|\lesssim\frac{1}{\#\mathcal{T}_{f_i}}\sum_{K\in\mathcal{T}_{f_i}}\|Q_{K}w_{11}\|_{0,f}\|1\|_{0,f}\\
            &\lesssim\sum_{K\in\mathcal{T}_{f_i}}(h^{-\frac{1}{2}}\|Q_{K}w_{11}\|_{0,K}+h^{\frac{1}{2}}\|Q_{K}w_{11}\|_{0,K})h^{\frac{1}{2}}\\
            &=\sum_{K\in\mathcal{T}_{f_i}}(\|Q_{K}w_{11}\|_{0,K}+h|Q_{K}w_{11}|_{1,K})\lesssim\sum_{K\in\mathcal{T}_{f_i}}\|Q_{K}w_{11}\|_{0,K}\\
            &\lesssim\|Qw_{11}\|_{0,\mathcal{T}_{f}}\lesssim\|Qw_{11}\|_{0,\tilde{K}}\le\|w_{11}\|_{0,\tilde{K}}\lesssim h|\tau|_{1,\tilde{K}},
        \end{aligned}
    \end{equation}
    and 
    \begin{equation}
        \label{eq:approx_psi_f_i}
        \|\tau_{z_iz_j}\|_{0,K}=\left\|\frac{6}{|z_iz_j|}(\lambda_i\lambda_j+\lambda_1\lambda_2\lambda_3(21\lambda_k-12))\right\|_{0,K}\lesssim 1.
    \end{equation}
    \eqref{eq:estimate_b_i} and \eqref{eq:approx_psi_f_i} mean $J_2\lesssim h|\tau|_{1,\tilde{K}}$.
    
    Third, by the Cauchy-Schwarz inequality,
    \begin{equation}
        \label{eq:estimate_c}
        |c|\lesssim\|w_{11}\|_{0,K}\|1\|_{0,K}\lesssim h\|w_{11}\|_{0,K}\lesssim h^2|\tau|_{1,K}\le h^2|\tau|_{1,\tilde{K}},
    \end{equation}
    and 
    \begin{equation}
        \label{eq:estimate_var_phi_K}
        \|\varphi_{K}\|_{0,K}=\left\|\frac{1}{|K|}\lambda_1\lambda_2\lambda_3(900\lambda_1-360\lambda_2-360\lambda_3)\right\|_{0,K}
        \lesssim h^{-1}.
    \end{equation}
    \eqref{eq:estimate_c} and \eqref{eq:estimate_var_phi_K} mean $J_3\lesssim h|\tau|_{1,\tilde{K}}$. Therefore:
    \begin{equation}
        \|I_{h}w_{11}\|_{0,K}\lesssim h|\tau|_{1,\tilde{K}}.
    \end{equation}
    Then we have:
    \begin{equation}
        \|\tau-\mathcal{I}_h\tau\|_0\le\sum_{K}(\|w\|_{K}+\|\mathcal{I}_hw\|_{K})\lesssim\sum_{K}h|\tau|_{1,\tilde{K}}\lesssim h|\tau|_{1,\Omega}.
    \end{equation}
    Estimating $|\mathcal{I}_hw|_{1}$ is nearly the same as $\|\mathcal{I}_h w\|_{0}$. The only differences are:
    \begin{equation}
        \begin{aligned}
            |\phi_{z_i}|_{1,K}\lesssim 1,\quad |\tau_{f_i}|_{1,K}\lesssim h^{-1},\quad |\varphi_{K}|_{1,K}\lesssim h^{-2}.
        \end{aligned}
    \end{equation}
    Then we can see $|\tau-\mathcal{I}_h\tau|_{1}\lesssim|\tau|_{1}$.

    For $\|I_{h}\tau\|_{1}$, since 
    \begin{equation}
        \|\mathcal{I}_h\tau-\tau\|_{1}\lesssim\|\mathcal{I}_h\tau-\tau\|_0+|\mathcal{I}_h\tau-\tau|_{1}\lesssim|\tau|_1,
    \end{equation}
    by Poincaré's inequality:
    \begin{equation}
        \|\mathcal{I}_h\tau\|_1\le\|\tau\|_1+\|\mathcal{I}_h\tau-\tau\|_1\lesssim\|\tau\|_1+|\tau|_1\lesssim|\tau|_1.
    \end{equation}

\end{proof}
